\documentclass[11pt]{article}
\usepackage{amsmath,amsfonts,amssymb,amsthm,mathrsfs,graphics,graphicx,mathabx,setspace,bbm,cleveref,xcolor,url,mathtools,booktabs}

\allowdisplaybreaks

\newcommand{\I}{{\bf 1}}
\newtheorem{proposition}{Proposition}[section]
\newtheorem{theorem}[proposition]{Theorem}
\newtheorem{corollary}[proposition]{Corollary}
\newtheorem{lemma}[proposition]{Lemma}
\newtheorem{remark}[proposition]{Remark}

\newcommand{\nc}{\newcommand}

\DeclareMathOperator{\vol}{V}
\nc{\BHd}{\mathbb{H}^d}
\nc{\R}{{\mathbb R}}
\nc{\BS}{{\mathbb S}}
\nc{\bS}{{\mathbb S}^{d-1}}
\nc{\N}{{\mathbb N}}
\nc{\BB}{{\mathbb B}}
\nc{\C}{{\mathbb C}}
\nc{\Z}{{\mathbb Z}}
\nc{\BP}{\mathbf{P}}
\nc{\BH}{\mathbb{H}}
\nc{\BQ}{\mathbf{Q}}
\nc{\bN}{{\mathbf N}}
\nc{\BX}{{\mathbb X}}
\nc{\BY}{{\mathbb Y}}
\nc{\cB}{{\mathcal B}}
\nc{\cE}{{\mathcal E}}
\nc{\cK}{{\mathcal K}}
\nc{\cC}{{\mathcal C}}
\nc{\cH}{{\mathcal H}}
\nc{\cL}{{\mathcal L}}
\nc{\cM}{{\mathcal M}}
\nc{\cR}{{\mathcal R}}
\nc{\cX}{{\mathcal X}}
\nc{\cY}{{\mathcal Y}}
\nc{\dint}{{\rm d}}
\nc{\D}{\Delta}
\nc{\g}{\gamma}
\nc{\cI}{\mathcal{I}}
\nc{\cZ}{\mathcal{Z}}
\nc{\cum}{\operatorname{cum}}
\nc{\sZ}{{\mathscr{Z}}}
\nc{\origin}{o}
\nc{\bI}{\mathbf{1}}
\nc{\1}{\bI}
\nc{\E}{\mathbf{E}}
\newcommand{\defeq}{\mathrel{\mathop:}=}
\nc{\sfp}{\mathsf{p}}

\DeclareMathOperator{\arcosh}{arcosh}
\nc{\Vol}{V_d}
\nc{\one}{\mathbbm{1}}

\nc{\Ih}{{\mathcal{I}_d}}

\usepackage{colortbl}

\numberwithin{equation}{section}

\makeatletter
\let\@fnsymbol\@alph
\makeatother

\begin{document}
\title{Visibility and intersection density \\for Boolean models in hyperbolic space}
\date{October 10, 2024}
\date{}

\renewcommand{\thefootnote}{\fnsymbol{footnote}}

\author{Tillmann B\"uhler\footnotemark[1],\, Daniel Hug\footnotemark[2] \, and Christoph Thäle\footnotemark[3]\,}

\footnotetext[1]{Institute of Stochastics, Karlsruhe Institute of Technology, tillmann.buehler@kit.edu}

\footnotetext[2]{Institute of Stochastics, Karlsruhe Institute of Technology, daniel.hug@kit.edu}

\footnotetext[3]{Faculty of Mathematics, Ruhr University Bochum, christoph.thaele@rub.de}

\maketitle

\begin{abstract}
For Poisson particle processes in hyperbolic space we introduce and study concepts analogous to the intersection density and the mean visible volume, which were originally {considered} in the analysis of Boolean models in Euclidean space. In particular, we determine a necessary and sufficient condition for the finiteness of the mean visible volume of a Boolean model in terms of the intensity and the mean surface area of the typical grain.

\medskip

\end{abstract}

\section{Introduction}

The Boolean model is a foundational concept in stochastic geometry that describes random spatial structures formed by the union of random shapes, called grains, which are placed at the points of a random point process. More precisely, the construction of a Boolean model usually begins with a point process, typically a Poisson point process, which defines the random locations in space where the grains are centered. Each grain, characterized by its shape, size, and orientation, is independently and randomly distributed according to a specified probability law on the space of compact subsets of the ambient space. By taking the union of all grains, the model generates a random set in space, capturing a rich variety of spatial patterns and structures. The Boolean model is particularly valued for its mathematical tractability and flexibility. Many key properties such as the {probability of complete coverage} or the mean volume fraction can be derived analytically. These features make the Boolean model a powerful tool for both theoretical studies and practical applications. We refer the reader to \cite{ChiuSKM,HLWSurvey,SW08}, for example.

The present paper investigates the intersection density and the mean visible volume in a Boolean model. In this context, the intersection density refers to the number of intersection points of the boundaries of any $d$ distinct particles with a given set, where $d$ denotes the dimension of the space in which the Boolean model is constructed. On the other hand, the mean visible volume is defined as the expected volume of the region that remains uncovered by the random grains and is visible from a fixed observation point, such as the origin, conditioned on the event that this point is not covered. Both quantities have been extensively studied for Boolean models in a $d$-dimensional \textit{Euclidean} space; see \cite[Section 4.6]{SW08} and \cite[Sections 4--6]{Wieacker}. In particular, the mean visible volume is always finite as soon as the grains are full-dimensional. In contrast, we focus on the intersection density and the mean visible volume of Boolean models in a $d$-dimensional \emph{hyperbolic} space $\BH^d$ of constant curvature $-1$; see Figure \ref{fig:BM} for a simulation. Also, here the Boolean model can be derived from a Poisson point process in $\BH^d$ with intensity $\gamma>0$, which is invariant under all isometries of $\BH^d$, but the probabilistic construction required for placing the grains at the points of the Poisson process is more involved. For this reason, alternative approaches have been proposed in \cite[Sections 2.2 and 2.3]{HLS2} and will be further developed here. Boolean models in hyperbolic space were studied in \cite{BenjaminiJonassonSchrammTykesson,HLS2,Lyons,Tykesson07,TykessonCalka} and give rise to a number of new phenomena compared to their Euclidean counterparts. In addition, the analysis of models in hyperbolic stochastic geometry often requires the development of new geometric tools, which are of independent interest. We will see examples for both effects also in the present text.

\begin{figure}[t]
    \centering
    \includegraphics[width=0.65\linewidth]{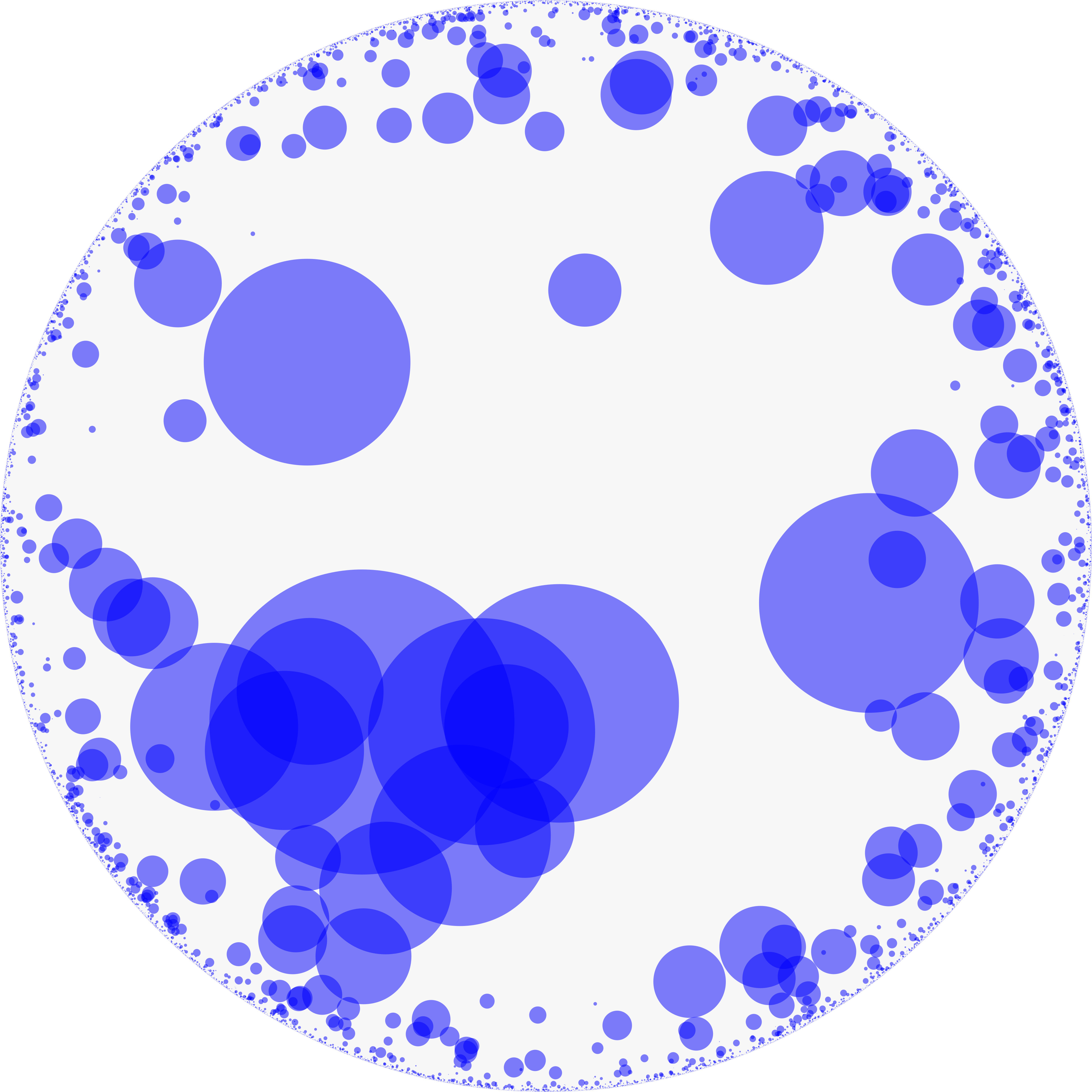}
    \caption{\small Simulation of a Boolean model in the hyperbolic plane whose grains are discs of random radius uniformly distributed in $[0,1]$. The saturation of a point represents the number of grains by which it is covered.}
    \label{fig:BM}
\end{figure}

It is known from \cite{Lyons} that the visibility properties of hyperbolic Boolean models exhibit a remarkable threshold phenomenon, which cannot be observed in Euclidean space. To be precise, the paper \cite{Lyons} deals with Boolean models in the $d$-dimensional hyperbolic space (and in fact more general spaces) whose grains are balls of fixed radius $R>0$. In this specific set-up, if we condition on the event that a fixed point $\sfp\in\BH^d$ is not covered by the grains of the Boolean model, then there exists a critical intensity parameter $\gamma_v=\gamma_v(R,d)>0$ such that:
\begin{itemize}
    \item for $\gamma<\gamma_v$ the grains are sparse  enough so that, {with positive probability}, one can find infinite geodesic rays emanating from $\sfp$ that remain unobstructed,

    \item for $\gamma>\gamma_v$ the grains become dense enough to almost surely obstruct all geodesics starting at $\sfp$.
\end{itemize}
The value for $\gamma_v$ is known and will be discussed in Remark \ref{rem:Visibility} below. More refined properties of the length of the geodesic rays starting at $\sfp$ are studied in \cite{TykessonCalka} (in dimension $d=2$).

We build upon these findings by exploring the \textit{mean visible volume} around a given point $\sfp$. Our primary result in this direction provides an explicit formula for the mean visible volume in terms of the intensity $\gamma$ and the mean surface content $v_{d-1}$ of the so-called typical grain. It applies to general space dimensions $d$ and general random hyperbolically convex grain shapes. A significant aspect of our result is the identification of a precise condition under which the mean visible volume remains finite. Specifically, we prove that the mean visible volume is finite if and only if the product $\gamma v_{d-1}$ exceeds the threshold value $(d-1)\frac{d\kappa_d}{\kappa_{d-1}}$, where $\kappa_j$ denotes the volume of the $j$-dimensional Euclidean unit ball. This condition reveals the interplay between the intensity of the underlying Poisson point process and the geometric properties of the grains that must be satisfied for a finite mean visible volume. 
In addition, we present an explicit formula for the intersection density of a $d$-dimensional hyperbolic Boolean model and establish a connection between this density and the mean visible volume. To further contextualize our findings, we draw a comparison with analogous results obtained for the zero cell of a stationary Poisson hyperplane tessellation in hyperbolic space.

\section{Preliminaries}

In this paper, we work in a $d$-dimensional hyperbolic space $\mathbb{H}^d$ of dimension $d\ge 2$ and with constant curvature $\kappa=-1$. We denote by $d_h(\,\cdot\,,\,\cdot\,)$ the hyperbolic distance, by $\cC^d$ the space of nonempty compact subsets of $\mathbb{H}^d$ and by $\cK^d$ the space of compact and (geodesically) convex subsets of $\mathbb{H}^d$ (including the empty set). The unimodular Lie group of isometries of $\BH^d$ is denoted by $\Ih$. It carries an (left and right) invariant Haar measure $\lambda$, which is unique up to normalization. We normalize $\lambda$ so that, for any $x\in \BH^d$, 
\begin{equation}\label{eq:normalize}
\int_{\Ih}{\bf 1}\{\varrho x\in\,\cdot\,\}\,\lambda(\dint\varrho)=\mathcal{H}^d(\,\cdot\,),
\end{equation}
where $\mathcal{H}^d$ denotes the $d$-dimensional Hausdorff measure constructed with respect to $d_h$. More generally, for $s\geq 0$ we denote by $\mathcal{H}^s$ the $s$-dimensional Hausdorff measure with respect to the hyperbolic distance $d_h$. Instead of $\mathcal{H}^d$, we often use $V_d$ to denote the volume functional on  compact convex sets in $\mathbb{H}^d$. For $A\in \cC^d$ and $x\in\mathbb{H}^d$, we set $d_h(A,x)\defeq\min\{d_h(a,x):a\in A\}$. Moreover, we set $d_h(\varnothing,x)\defeq\infty$.

For $A\in\cK^d$  and $r\ge 0$, let $\BB(A,r)=\{x\in\BHd\colon d_h(A,x)\le r\}$
 be the parallel set of  $A$  at distance $r$. For $j\in\{0,\ldots,d-1\}$ and $r\ge 0$, let
$$
l_{d,j}(r)\defeq\binom{d-1}{j}\int_0^r \cosh^j(t)\sinh^{d-1-j}(t)\, \dint t.
$$
We note that the functions $l_{d,0},\ldots,l_{d,d-1}$ are linearly independent in the real vector space of continuous functions on $[0,\infty)$.  
 The Steiner formula \eqref{eq:Steiner} involves the existence of (uniquely determined) continuous, isometry invariant, and additive functionals $V_0,\ldots,V_{d-1}$ on $\cK^d$  
such that
\begin{align}\label{eq:Steiner}
V_d(\BB(A,r))=V_d(A)+\sum_{j=0}^{d-1} V_j(A)\cdot l_{d,j}(r),\quad r\ge 0;
\end{align}
a more general (local) version of \eqref{eq:Steiner} is provided in \cite[Theorem 2.7]{Kohlmann} (note that the result in \cite{Kohlmann} is stated in $\mathbb{H}^{d+1}$ and the coefficient functions are normalized differently). The normalization in \eqref{eq:Steiner} is chosen as in \cite[Equation (3.3)]{HLS2} and consistent with the normalization in \cite[Equation (4.12)]{Schneider2014} in the Euclidean framework. It is the natural choice of normalization, if $V_j(A)$ is considered as a curvature integral involving a normalized elementary symmetric function (of order $d-1-j$) of principal curvatures of $\partial A$; see \cite[Equation (1.11)]{Sol} for the case of parallel sets of smooth  hypersurfaces in hyperbolic space. Note that $V_j(\varnothing)=0$ for $j\in\{0,\ldots,d\}$. 
 
If $A$ is convex and has nonempty interior, then
\begin{equation}\label{eq3.2Vnm1}
V_{d-1}(A)=\mathcal{H}^{d-1}(\partial A).
\end{equation}
We note that in the Euclidean case the corresponding relation involves an additional factor $1/2$ on the right-hand side, due to the different normalization of $V_{d-1}$. However, if $A\subset \BHd$ is compact, convex, and has dimension at most $d-1$, then $V_{d-1}(A)=2\mathcal{H}^{d-1}(\partial A)$. If the dimension of $A$ is at most $d-2$, then both sides of \eqref{eq3.2Vnm1} vanish. Moreover, $V_0(A)$ is the (total) Gauss curvature integral of $A$ (which follows as a special case of \cite[Theorem 2.7 and (2.7)]{Kohlmann}). In Euclidean space the normalization of $V_0$ is chosen such that $V_0$ is equal to the Euler characteristic $\chi$, whereas in hyperbolic space the relation between $V_0$ and the Euler characteristic $\chi$ is non-trivial and will be described below. 

A subset $A\subseteq\mathbb{H}^d$ is said to be $k$-rectifiable, for some $k\in\{0,\ldots,d\}$, if $k=0$ and $A$ is finite or if $k\ge 1$ and $A$ is the image of a bounded subset of $\R^k$ under a Lipschitz map. We say that $A$ is Hausdorff $k$-rectifiable if $\mathcal{H}^k(A)<\infty$ and there exist $k$-rectifiable subsets $B_1,B_2,\ldots$ of $\BH^d$ such that $\mathcal{H}^k\left(A\setminus \bigcup_{i\ge 1}B_i\right)=0$. For these definitions, see also \cite[Section 3.2.14]{Federer69} and \cite[Definitions 5.12 and 5.13]{Brothers66}. In particular, every bounded subset of $\BH^d$ is $d$-rectifiable.

\section{Hyperbolic intrinsic volumes}

It is a well-known fact that the Euclidean intrinsic volumes are intrinsic in the sense that they do not depend on the dimension of the ambient space.
That is, if \(1\le k \leq d\), $A\subset\R^k$ is a  compact convex set and \(\iota \colon \R^k \to \R^d\) is an isometric embedding of \(\R^k\) into \(\R^d\) (such as the one mapping $x\in\R^k$ to $(x,0,\ldots,0)$ in $\R^d$), then
\[ V^{\R^d}_j(\iota(A)) = \begin{cases}
    V^{\R^k}_j(A), & j=0,\ldots,k, \\
    0,& j= k+1,\ldots,d,
\end{cases} \]
where \(V^{\R^d}_j\) and \(V^{\R^k}_j\) denote the \(j\)-th intrinsic volume of a set measured in \(\R^d\) and \(\R^k\), respectively.
For this reason, they are called \emph{intrinsic} volumes, and the ambient space is usually omitted in the notation (for details, see \cite[{Equations} (4.1) and (4.22)]{Schneider2014} or \cite[pp.~106-108]{HugWeil2020}).

With a different normalization, the same property holds for the hyperbolic intrinsic volumes as well.
Define for $j=0,\ldots,d-1$ the functions \(\ell_{d,j}\) by
\[ \ell_{d,j}(r) \defeq \omega_{d-j} \int_0^r \cosh^j(t) \sinh^{d-1-j}(t) \,\dint t,\qquad r\geq 0, \]
where $ \omega_i=2 \pi^{i/2}/\Gamma(i/2)=i\kappa_i$ denotes the surface area of the \(i\)-dimensional {Euclidean} unit ball for $i\in\N$. 
For \(A \in \cK^d\)  put \(V^{\BHd}_d(A) \defeq \mathcal{H}^d(A)\) and define the intrinsic volumes \(V^{\BHd}_0,\ldots,V^{\BHd}_{d-1}\) as the coefficients in the Steiner formula
\begin{equation}\label{eq:Steinerhyp}
\mathcal{H}^d(\BB(A,r)) = \mathcal{H}^d(A) + \sum_{j=0}^{d-1} V^{\BHd}_j(A)\cdot \ell_{d,j}(r) ,\qquad r\geq 0. 
\end{equation}
Note that \eqref{eq:Steinerhyp} remains true for $d=1$ with $V_0^{\mathbb{H}^1}(A)=1$ if $A\in \cK^1$ satisfies $A\neq\varnothing$.

In the following, a $k$-dimensional totally geodesic subspace of $\mathbb{H}^d$ is called a $k$-plane for $1\le k\le d-1$.
We denote the space of all \(k\)-planes by \(A_h(d,k)\).

\begin{proposition}\label{prop:intrinsic_hyp}
    Let \(1 \leq k \leq d-1\) and identify \(\BH^k\) with an arbitrary \(L_k \in A_h(d,k)\) via an isometry \(\iota \colon \BH^k \to \BH^d\) such that $\iota(\BH^k)=L_k$. 
    If \(A \in \cK^k\), then
    \[ V^{\BH^d}_j(\iota(A)) = \begin{cases}
    V^{\BH^k}_j(A), & j=0,\ldots,k, \\
    0,& j\ge  k+1 .
\end{cases} \]
\end{proposition}

\begin{remark}\label{rem:propintrinsic} 
{\rm For $A\in\cK^d$ and $d\ge 1$, a comparison of \eqref{eq:Steiner} and \eqref{eq:Steinerhyp} yields that
$$
\omega_{d-j}V_j^{\mathbb{H}^d}(A)=\binom{d-1}{j}V_j(A), \quad j= 0,\ldots,d-1,
$$
and \(V_d^{\mathbb{H}^d}(A) = V_d(A)\).
If $I\subset \cK^d$ is a geodesic segment, then Proposition \ref{prop:intrinsic_hyp} implies that 
\begin{alignat}{2}
V_0^{\mathbb{H}^d}(I) & = 1, \quad & V_1^{\mathbb{H}^d}(I) & = \mathcal{H}^1(I), \nonumber
\shortintertext{and therefore}
V_0(I) & = \omega_d, \quad & V_1(I) & = \kappa_{d-1} \mathcal{H}^1(I).
\nonumber
\end{alignat}
}
\end{remark}

In addition to the Steiner formula, the proof of Proposition \ref{prop:intrinsic_hyp} requires two integration formulas in hyperbolic space: an orthogonal disintegration formula that can be seen as a hyperbolic analogue of Fubini's theorem and the hyperbolic polar integration formula.

\begin{lemma}\label{lem:IntegrationFormulas}
\begin{itemize}
    \item[{\rm (a)}] If \(1 \leq k \leq d-1\) and \(L_k \in A_h(d,k)\) is a fixed $k$-plane, then 
    \[ \int_{\BHd} f(x) \,\cH^d(\dint x) = \int_{L_k} \int_{L_k^\perp(y)} f(z)\,\cosh^k(d_h(z,L_k))  \,\cH^{d-k}(\dint z) \,\cH^k(\dint y), \]
    for any measurable function \(f \colon \BHd \to [0,\infty]\),  
    where \(L_k^\perp(y)\) is the unique \((d-k)\)-plane  orthogonal to \(L_k\) that intersects \(L_k\) in \(y\) and $d_h(z,L_k)=d_h(z,y)$ stands for the geodesic distance of $z$ to $L_k$.

    \item[{\rm (b)}] If \(d \geq 1\) and  \(f \colon \BHd \to [0,\infty]\) is measurable, then
    \[ \int_{\BHd} f(x) \,\cH^d(\dint x) = \int_{\BS^{d-1}_\sfp} \int_0^\infty f(\exp_\sfp(tu)) \sinh^{d-1}(t) \,\dint t \,\cH_\sfp^{d-1}(\dint u), \]
    where \(\BS^{d-1}_\sfp\) and \(\cH_\sfp^{d-1}\) refer to the unit sphere and the \((d-1)\)-dimensional Hausdorff measure in the tangent space of $\mathbb{H}^d$ at \(\sfp\), respectively.
\end{itemize}
\end{lemma}
\begin{proof}
To derive part (a) we use the method established in Section 5 of \cite{BHT} and translate the problem into a problem in Euclidean space. For that purpose we work with the Beltrami--Klein (projective) model of $\mathbb{H}^d$, which is given by the open $d$-dimensional Euclidean unit ball $\mathsf{B}^d$. Applying \cite[Lemma 5.1]{BHT} with $k=d$ there gives
\begin{align*}
    \int_{\BH^d}f(x)\,\cH^d(\dint x) = \int_{\mathsf{B}^d}f(x)(1-\|x\|^2)^{-\frac{d+1}{ 2}}\,\cH_{\mathbb{R}^d}^d(\dint x),
\end{align*}
where $\|\,\cdot\,\|$ stands for the Euclidean norm and $\cH_{\mathbb{R}^d}^d$ for the $d$-dimensional Hausdorff measure in $\mathbb{R}^d$. Since we are working with the projective model of $\BH^d$, a $k$-dimensional totally geodesic subspace $L_k\subseteq\BH^d$ is the intersection $\widetilde{L}_k\cap\mathsf{B}^d$ of a $k$-dimensional affine subspace $\widetilde{L}_k\subseteq \R^d$ with $\mathsf{B}^d$. Since each $k$-dimensional totally geodesic subspace of $\BH^d$ can be mapped to any other by an isometry of $\BH^d$, we may assume without loss of generality that $\widetilde{L}_k$ passes through the center of $\mathsf{B}^d$. Then, using Fubini's theorem and denoting by $\widetilde{L}_k^\perp$ the $(d-k)$-dimensional Euclidean orthogonal complement of $\widetilde{L}_k$, it follows that
\begin{align*}
    \int_{\mathsf{B}^d}f(x)(1-\|x\|^2)^{-\frac{d+1}{2}}\,\cH_{\mathbb{R}^d}^d(\dint x) = \int_{\widetilde{L}_k\cap\mathsf{B}^d}\int_{(\widetilde{L}_k^\perp+y)\cap\mathsf{B}^d} f(z)(1-\|z\|^2)^{-\frac{d+1}{ 2}}\,\cH_{\mathbb{R}^d}^{d-k}(\dint z)\,\cH_{\mathbb{R}^d}^k(\dint y).
\end{align*}
To the inner integral we now apply \cite[Lemma 5.1]{BHT} again to see that the last expression is the same as
\begin{align*}
    &\int_{\widetilde{L}_k\cap\mathsf{B}^d}\int_{L_k^\perp(y)}f(z)(1-\|z\|^2)^{-\frac{d+1}{2}}(1-\|z\|^2)^{\frac{d-k+1}{2}}(1-\|y\|^2)^{-\frac{1}{2}}\,\cH^{d-k}(\dint z)\,\cH_{\mathbb{R}^d}^k(\dint y)\\
    &\quad=\int_{\widetilde{L}_k\cap\mathsf{B}^d}\int_{L_k^\perp(y)}f(z)(1-\|z\|^2)^{-\frac{k}{2}}(1-\|y\|^2)^{-\frac{1}{2}}\,\cH^{d-k}(\dint z)\,\cH_{\mathbb{R}^d}^k(\dint y),
\end{align*}
where we also used the fact that $\widetilde{L}_k+y$ has Euclidean distance $\|y\|$ from the center of $\mathsf{B}^d$; see also the argument in the second paragraph of the proof of \cite[Lemma 5.3]{BHT} concerning orthogonality of subspaces in the Beltrami--Klein model. Writing $\bullet$ for the Euclidean scalar product in $\R^d$ and using \cite[Theorem 6.1.1]{Ratcliffe}, we obtain
$$
\cosh^2 d_h(z,L_k)=\cosh^2 d_h(z,y)=\frac{(1-z\bullet y)^2}{(1-\|z\|^2)(1-\|y\|^2)}=\frac{1-\|y\|^2}{1-\|z\|^2}.
$$
Plugging this into the last integral expression and applying once more \cite[Lemma 5.1]{BHT} to the outer integral, we arrive at
\begin{align*}
    \int_{\BH^d}f(x)\,\cH^d(\dint x) 
    &=\int_{\widetilde{L}_k\cap\mathsf{B}^d}\int_{L_k^\perp(y)}f(z)\,\cosh^k(d_h(z,L_k))\,\cH^{d-k}(\dint z)\,(1-\|y\|^2)^{-\frac{k+1}{2}}\,\cH_{\mathbb{R}^d}^k(\dint y)\\
    &=\int_{L_k}\int_{L_k^\perp(y)}f(z)\,\cosh^k(d_h(z,L_k))\,\cH^{d-k}(\dint z)\,\cH^k(\dint y),
\end{align*}
which concludes the proof of part (a). Part (b) is a consequence of Equations (3.16) and (3.22) in \cite{Chavel1993}.
\end{proof}

\begin{remark}
{\rm The special case $k=1$ of Lemma \ref{lem:IntegrationFormulas} can be found as Lemma 2.1 in \cite{RosenSchulteThaeleTrapp}. It was derived there using different methods such as the coarea formula.}
\end{remark}

We also need the following identity.

\begin{lemma}\label{prop:ell_identity}
    If \(0 \leq j < k \le d-1\), then
    \begin{equation}
    \label{eq:ell_identity}
        \int_0^r \ell_{d,k}\left(\arcosh\left(\frac{\cosh(r) }{\cosh(s)}\right)\right) \ell_{k,j}'(s)\, \dint s = \ell_{d,j}(r),\qquad r\geq 0.
    \end{equation}
\end{lemma}
\begin{proof} 
    We start by collecting some auxiliary facts.
    The beta function, which is  defined by \(B(a,b) := \int_0^1 t^{a-1} (1-t)^{b-1}\,\dint t\)
    for \(a,b > 0\), can be expressed in terms of the gamma function as \(B(a,b) = \Gamma(a)\Gamma(b)/\Gamma(a+b)\), hence 
    \begin{equation}
    \label{eq:rel_beta_omega}
        \frac{1}{2}B\left(\frac{j}{2},\frac{k}{2}\right) = \frac{\omega_{j+k}}{\omega_j \omega_k},
    \end{equation}
    for \(j,k \in \N\). We will also make use of the integral identity
    \begin{equation}
    \label{eq:beta_identity}
        \int_0^1 \frac{t^{a-1} (1-t)^{b-1}}{(1+ct)^{a+b}} \,\dint t = (1+c)^{-a} B(a,b),
    \end{equation}
    which holds for \(a,b>0\) and \(c>-1\).
    This can be seen by rewriting the integral as
    \[ \int_0^1 (1+c)^{-a}  \left(\frac{(1+c)t}{1+ct}\right)^{a-1} \left(\frac{1-t}{1+ct}\right)^{b-1} \frac{1+c}{(1+ct)^2}\,\dint t, \]
    and substituting \(y = (1+c)t/(1+ct)\).

    We now turn to the proof of the statement of the lemma. For this, let $f(r,s)$ denote the integrand on the left-hand side of  \eqref{eq:ell_identity}, where $(r,s)\in [0,\infty)^2$ with $s\le r$, and let  $F(r):=\int_0^rf(r,s)\, \dint s$ for $r\ge 0$. 
    Both sides of \eqref{eq:ell_identity} are zero for \(r=0\). We show that they are  continuous as functions of $r\ge 0$ and that  their right derivatives agree, which yields the assertion of the lemma. We first consider $F$. Let $r\ge 0$ be fixed. If $h\in (0,1]$, then
    $$
    F(r+h)-F(r)=\int_0^r f(r+h,s)-f(r,s)\,\dint s + \int_r^{r+h} f(r+h,s)\, \dint s.
    $$
    Since for any given $\varepsilon>0$, we have $|f(r+h,s)|\le \varepsilon$ if $s\in [r,r+h]$ and $h>0$ is small enough (recall that $r\ge 0$ is fixed), it follows that 
    $$
    \lim_{h\downarrow 0}\frac{1}{h}\int_r^{r+h} f(r+h,s)\, \dint s=0.
    $$
    Moreover, there are $\theta=\theta_{r,s,h}\in(0,1)$ such that
    $$
    \int_0^r \frac{f(r+h,s)-f(r,s)}{h}\,\dint s
    =\int_0^r\partial _rf(r+\theta h,s)\, \dint s,
    $$
    where we write $\partial_rf$ for the partial derivative of $f$ with respect to the first argument and 
    $$
    \partial_rf(t,s)=\omega_{d-k} \omega_{k-j} \cosh^k(t) \sinh(t)\frac{\sinh^{k-1-j}(s)}{\cosh^{d-1-j}(s)} (\sinh^2(t) - \sinh^2(s))^{\frac{d-2-k}{2}},\quad 0\le s<t.
    $$
   Recall that $0<h\le 1$.  If $k\le d-2$, then  $|\partial _rf(r+\theta h,s)|$ is bounded from above as a function of $s\in (0,r)$ by a constant that depends only on $r$ and is independent of $h$. If $k=d-1$, then
   $$
    |\partial_rf(r+\theta h,s)| \leq \omega_{d-k} \omega_{k-j} \cosh^k(r+1) \sinh^{k-j}(r+1)\left(\sinh^2(r) - \sinh^2(s)\right)^{-\frac{1}{2}},\quad 0<s<r.
   $$
   The subsequent calculation shows that as a function of $s\in (0,r)$, the right-hand side is integrable (and independent of $h$). Hence, we can interchange the order of the limit $h\downarrow 0$ and the integration to get, for the right derivative $\partial_r^+F(r)$ of $F$ at $r$, 
\begin{align*}
    \partial_r^+F(r)&=\int_0^r \lim_{h\downarrow 0}\partial _rf(r+\theta h,s)\,\dint s\\
        &= \omega_{d-k} \omega_{k-j} \cosh^k(r) \sinh(r)
            \int_0^r \frac{\sinh^{k-1-j}(s)}{\cosh^{d-1-j}(s)} (\sinh^2(r) - \sinh^2(s))^\frac{d-2-k}{2}\, \dint s\\
        &= \omega_{d-k} \omega_{k-j} \cosh^k(r) \sinh^2(r)
            \int_0^1 \frac{(x \sinh(r))^{k-1-j}}{(1 + x^2 \sinh^2(r))^{\frac{d-j}{2}}}
            (\sinh^2(r) - x^2 \sinh^2(r))^{\frac{d-2-k}{2}} \,\dint x\\
        &= \omega_{d-k} \omega_{k-j} \cosh^k(r) \sinh^{d-1-j}(r)
            \int_0^1 \frac{x^{k-1-j}}{(1 + x^2 \sinh^2(r))^\frac{d-j}{2}} (1-x^2)^\frac{d-2-k}{2} \,\dint x\\
        &= \omega_{d-k} \omega_{k-j} \cosh^k(r) \sinh^{d-1-j}(r) \frac{1}{2}
            \int_0^1 \frac{t^{\frac{k-j}{2} - 1} (1-t)^{\frac{d-k}{2}-1}}{(1 + t \sinh^2(r))^\frac{d-j}{2}} \,\dint t\allowdisplaybreaks\\
        &= \omega_{d-k} \omega_{k-j} \cosh^k(r) \sinh^{d-1-j}(r)
            \frac{1}{(1+\sinh^2(r))^\frac{k-j}{2}} \frac{1}{2} B\left(\frac{k-j}{2}, \frac{d-k}{2}\right)\allowdisplaybreaks\\
        &\quad= \omega_{d-j} \cosh^j(r) \sinh^{d-1-j}(r),
    \end{align*}
    where we used the substitution \(\sinh(s) = x \sinh(r)\) in the third line,
    the substitution \(t = x^2\) in the fifth  line, \eqref{eq:beta_identity} in the sixth  line, 
    and \eqref{eq:rel_beta_omega} in the last line.
    We also used the identity $\cosh^2(u) - \sinh^2(u) = 1$ for $u\in\mathbb{R}$ repeatedly. This shows in particular that $F$ is continuous from the right. Continuity from the left is seen from the decomposition 
$$
F(r+h)-F(r)=\int_0^{r+h} f(r+h,s)-f(r,s)\,\dint s-\int_{r+h}^r f(r,s)\, \dint s
$$
for $h<0$ and $r+h\ge 0$.

    Differentiating the right-hand side gives
    \[ \ell_{d,j}'(r) = \omega_{d-j} \cosh^j(r) \sinh^{d-1-j}(r), \]
    as well. This concludes the proof.
\end{proof}

We can now turn to the proof of the main statement of this section.

\begin{proof}
[Proof of \Cref{prop:intrinsic_hyp}]
Let \(1 \leq k \leq d-1\).
We fix some nonempty \(A \in \cK^k\).
For \(x \in \BHd\), let \(\pi_{L_k}(x)\) be the orthogonal projection of \(x\) onto \(L_k\).
By the Pythagorean theorem of hyperbolic geometry (compare, e.g.\ \cite[Theorem 3.5.3]{Ratcliffe}), it holds that
\begin{equation}
\label{eq:dist_pythagoras}
    \cosh(d_h(x, \iota(A))) = \cosh(d(x, \pi_{L_k}(x))) \cdot \cosh(\pi_{L_k}(x), \iota(A)). 
\end{equation}
Since \(\pi_{L_k}(z) = y\) for all \(z \in L_k^\perp(y)\), it follows from  \eqref{eq:dist_pythagoras}, Lemma \ref{lem:IntegrationFormulas} (a) and  Lemma \ref{lem:IntegrationFormulas} (b) (with \(L_k^\perp(y) \cong \BH^{d-k}\)) that, for \(r \geq 0\),
\begin{align*}
    &\cH^d(\BB(\iota(A),r))\\
    &\quad= \int_{\BHd} \one\{ d_h(x, \iota(A)) \leq r\} \,\cH^d(\dint x)\\
    &\quad= \int_{L_k} \int_{L_k^\perp(y)} \cosh^k(d_h(y,z)) \one\{ \cosh(d_h(z, y)) \cdot \cosh(d_h(y, \iota(A))) \leq \cosh(r) \}\, \cH^{d-k}(\dint z)\, \cH^k(\dint y)\\
    &\quad= \int_{L_k} \omega_{d-k} \int_0^\infty \sinh^{d-k-1}(t) \cosh^k(t) \one\{ \cosh(t) \cdot \cosh(d_h(y, \iota(A))) \leq \cosh(r) \} \,\dint t \,\cH^k(\dint y)\\
    &\quad= \int_{\BH^k} \omega_{d-k} \int_0^{\infty} \sinh^{d-k-1}(t) \cosh^k(t) \one\{ \cosh(t) \cdot \cosh(d_h(y, A)) \leq \cosh(r) \} \, \dint t \,\cH^k(\dint y),
\end{align*}
where we used in the last step the fact that \(\iota \colon \BH^k \to \BH^d\) is an isometry with $\iota(\BH^k)=L_k$.
We write the integrand as \(f( d_h(y,A))\), where
\[ 
f(s) \defeq \omega_{d-k} \int_0^{\infty} \sinh^{d-k-1}(t) \cosh^k(t) \one\{ \cosh(t) \cdot \cosh(s) \leq \cosh(r) \} \, \dint t, \qquad s \geq 0.  
\]
Note that \(f(s) = \ell_{d,k}(\arcosh\left({\cosh(r)}/{\cosh(s)}\right))\) for \(s \leq r\) and \(f(s)=0\) otherwise. Consequently,
\begin{align}\label{eq:dh10}
    \cH^d(\BB(\iota(A),r))
    &= \int_A f(0) \,\cH^k(\dint y) + \int_{\BB(A,r)\setminus A} f(d_h(y,A)) \,\cH^k(\dint y).
\end{align}
The first integral is equal to
\begin{equation}\label{eq:dh11} 
\int_A f(0)\, \cH^k(\dint y) = \cH^k(A) \ell_{d,k}(r) = V^{\BH^k}_k(A) \ell_{d,k}(r). 
\end{equation}
To evaluate the second integral, define \(g(s) \defeq \cH^k(\BB(A,s))\), \(s \geq 0\), where $\BB(A,s)$ is the parallel set of $A$ in $\mathbb{H}^k$.
By the Steiner formula, \(g\) is continuously differentiable with derivative given by
\begin{equation}\label{eq:dh1a} g'(s) = \sum_{j=0}^{k-1} \ell_{k,j}'(s)V^{\BH^k}_j(A). \end{equation}
We claim that
\begin{equation}\label{eq:dh1b} 
\int_{\BB(A,r)\setminus A} f(d_h(y,A)) \,\cH^k(\dint y) = \int_0^r f(s) g'(s)\, \dint s. 
\end{equation}
To see this, take an arbitrary partition \(s_0=0,\ldots,s_n=r\) of \([0,r]\) and pick \(\xi_i \in (s_{i-1},s_i)\), so that \(g(s_i)-g(s_{i-1})=g'(\xi_i)(s_i - s_{i-1})\), \(i=1,\ldots,n\).
We can now write
\begin{align*}
    &\int_{\BB(A,r)\setminus A} f(d_h(y,A))\, \cH^k(\dint y)\\
    &\quad=\sum_{i=1}^n \int_{\BB(A,s_{i})\setminus \BB(A,s_{i-1})} f(d_h(y,A))\, \cH^k(\dint y)\\
    &\quad= \sum_{i=1}^n f(\xi_i) (g(s_i) - g(s_{i-1})) + \sum_{i=1}^n \int_{\BB(A,s_{i})\setminus \BB(A,s_{i-1})} f(d_h(y,A))-f(\xi_i) \,\cH^k(\dint y)\\
    &\quad= \sum_{i=1}^n f(\xi_i) g'(\xi_i) (s_i - s_{i-1}) + \sum_{i=1}^n \int_{\BB(A,s_{i})\setminus \BB(A,s_{i-1})} f(d_h(y,A))-f(\xi_i) \,\cH^k(\dint y).
\end{align*}
As the mesh \(\max_{1\leq i\leq n} |s_i-s_{i-1}|\) of the partition goes to zero, the first sum converges to \(\int_0^r f(s) g'(s)\, \dint s\) and the second sum converges to zero, since \(f\) is uniformly continuous on \([0,r]\). This proves \eqref{eq:dh1b}. 
Plugging \eqref{eq:dh1a} into \(\int_0^r f(s) g'(s) \,\dint s\), we obtain
\begin{equation}\label{eq:dh1c} 
 \int_0^r f(s) g'(s)\, \dint s=\sum_{j=0}^{k-1} V^{\BH^k}_j(A) \cdot \int_0^r f(s) \ell_{k,j}'(s)\, \dint s 
    = \sum_{j=0}^{k-1} V^{\BH^k}_j(A) \cdot \ell_{d,j}(r), \end{equation}
where we used Lemma \ref{prop:ell_identity}.
Combining \eqref{eq:dh10}, \eqref{eq:dh11}, \eqref{eq:dh1b} and \eqref{eq:dh1c}, we obtain
\[ \cH^d(\BB(\iota(A),r)) = \sum_{j=0}^{k} V^{\BH^k}_j(A) \cdot \ell_{d,j}(r). \]
On the other hand, applying the Steiner formula directly yields
\[ \cH^d(\BB(\iota(A),r)) = \cH^d(\iota(A)) + \sum_{j=0}^{d-1} V^{\BH^d}_j(\iota(A)) \cdot \ell_{d,j}(r). \]
The claim now follows from the preceding two identities, since the functions \(\ell_{d,0},\ldots,\ell_{d,d-1}\) are linearly independent.
\end{proof}

\section{Stationary Poisson particle processes in hyperbolic space}\label{sec:ParticleProcesses}

 Let $\eta$ be a stationary Poisson particle process on $\cC^d$. In our set-up stationarity refers to the invariance of the distribution of $\eta$ under isometries of $\mathbb{H}^d$. We refer to \cite{HLS2} for an introduction to stationary Poisson particle processes in hyperbolic space and to stationary Boolean models derived from them. This means that the intensity measure $\Lambda:=\E\eta$ of $\eta$ is locally finite, diffuse by \cite[Lemma 2.4]{HLS2}, and hence $\eta$ is a simple particle process. 

Let $c:\cC^d\to\BH^d$ be an isometry covariant center function, where the condition of isometry covariance means that $c(\varrho A)=\varrho c(A)$ for all $A\in\cC^d$ and $\varrho\in\Ih$.  By \cite[Lemma 2.4 and Theorem 2.3]{HLS2}, the intensity 
$$
\gamma := \E\sum_{K\in\eta}\I\{c(K)\in B\},\qquad B\subset\BH^d\text{ measurable with }\cH^d(B)=1,
$$
is independent of the center function and finite. In addition, we suppose that $\gamma$ is also positive. Moreover, it is shown in \cite[Theorem 2.1]{HLS2} that there exists a unique   probability measure
$\BQ$ on $\cC^d$ which is concentrated
on $\cC_\sfp^{d}=\{A\in\cC^d:c(A)=\sfp\}$, invariant under isometries fixing $\sfp$, and satisfies
\begin{equation}\label{eqdisintegration}
\Lambda(\,\cdot\,) =\gamma \int_{\cC^{d}}\int_{\Ih} \I\{\varrho K\in\,\cdot\,\} \,\lambda(\dint \varrho) \,\BQ(\dint K).
   \end{equation}
The probability measure $\BQ$ is called the distribution of the typical grain of $\eta$. 

We complement the introduction of stationary Poisson particle processes and Boolean models in \cite[Section 2]{HLS2} by explaining how the marking of a stationary Poisson point process in $\BH^d$ can be used for an alternative approach. Let $\xi$ be a stationary Poisson point process in $\BH^d$ with intensity measure $\gamma\cdot \mathcal{H}^d$, where $\gamma>0$. We adopt the notation from \cite[Section 2]{HLS2} and consider for each $x\in\BH^d$ the set $\Ih(\sfp,x):=\{\varrho \in \Ih:\varrho\sfp=x\}$ and the probability kernel $\kappa$ from $\BH^d$ to $\Ih$ introduced in \cite[Equation (2.6)]{HLS2}; see also Equation \eqref{eq:KappaDef} below.
For $x\in \BH^d$, the probability measure $\kappa(x,\,\cdot\,)$ on $\Ih$ is concentrated on $\Ih(\sfp,x)$, and the kernel $\kappa$ disintegrates the Haar measure $\lambda$ on $\Ih$ so that
\begin{align}\label{disintlambda}
\int_{\mathbb{H}^d}\int_{\Ih} \I\{\varrho\in\,\cdot\,\}\,\kappa(x,\dint \varrho)\,\mathcal{H}^d(\dint x)=\lambda(\,\cdot\,).
\end{align}  
We now define a kernel $\alpha$ from $\BH^d$ to $\cC^d$ (see \cite[p.~40]{LP18}) by  
$$
\alpha(x,\,\cdot\,):=\int_{\cC^d}\int_{\Ih}\I\{\varrho K\in\,\cdot\,\}\, \kappa(x,\dint \varrho)\, \BQ(\dint K),\quad x\in\BH^d.
$$
Since $\BQ$ is concentrated on  $\cC_\sfp^{d}$ and the center function is isometry covariant, the probability measure $\alpha(x,\cdot)$ is concentrated on $\cC_{x}^{d} = \{A\in\cC^d:c(A)=x\}$. 
Let $\Psi$ be an $\alpha$-marking of $\xi$, as defined in \cite[Definition 5.3]{LP18}.  Then $\Psi$ is a (marked) Poisson process on $\BH^d\times\cC^d$ with intensity measure
$$
(\E\Psi)(\,\cdot\,)=\gamma\int_{\BH^d}\int_{\cC^d}\int_{\Ih}\I\{(x,\varrho K)\in\,\cdot\,\}\,
\kappa(x,\dint\varrho)\, \BQ(\dint K)\, \mathcal{H}^d(\dint x);
$$
see the marking theorem \cite[Theorem 5.6]{LP18}. The mapping lemma \cite[Theorem 5.1]{LP18} for Poisson processes shows that $\widetilde{\eta} = \Psi(\BH^d\times\,\cdot\,)$ is a Poisson process on $\cC^d$ with intensity measure
\begin{align*}
(\E\widetilde{\eta})(\,\cdot\,)&=\gamma\int_{\cC^d}\int_{\BH^d}\int_{\Ih}\I\{ \varrho K\in\,\cdot\,\}\,
\kappa(x,\dint\varrho)\,\mathcal{H}^d(\dint x)\, \BQ(\dint K)\\
&=\gamma\int_{\cC^d}\int_{\Ih}\I\{ \varrho K\in\,\cdot\,\}\, \lambda(\dint \varrho)\, \BQ(\dint K)=\Lambda(\,\cdot\,),   
\end{align*}
where \eqref{disintlambda} was used in the last step. Thus, $\eta\stackrel{d}{=}\widetilde{\eta}$ can be obtained by (position dependent) marking with the kernel $\alpha$ from a stationary Poisson point process in $\BH^d$.

There is yet another representation of a stationary Poisson particle process, which avoids the kernel $\kappa$. To present it, for each $x\in\BH^d$ we select an isometry $\varrho_x \in\Ih(\sfp, x)$ in such a way that the mapping $x\mapsto\varrho_x$ is measurable. Such an assignment exists as explained on page 121 of \cite{Ratcliffe}. Then, the measure \(\kappa(x,\,\cdot\,)\) can be represented as
\begin{equation}\label{eq:KappaDef}
\kappa(x,\,\cdot\,) = \int_{\Ih(\sfp,\sfp)} \I\{\varrho_x  \varrho \in \,\cdot\,\} \,\kappa(\sfp,\dint \varrho).
\end{equation}
Moreover, the kernel \(\alpha\) can then be rewritten as
    \begin{align*}
        \alpha(x,\,\cdot\,)
            &= \int_{\cC^d}\int_{\Ih}\I\{\varrho K\in\,\cdot\,\}\, \kappa(x,\dint \varrho)\, \BQ(\dint K)\\
            &= \int_{\cC^d}\int_{\Ih(\sfp,\sfp)}\I\{\varrho_x \varrho K\in\,\cdot\,\}\, \kappa(\sfp,\dint \varrho)\, \BQ(\dint K)\\
            &= \int_{\Ih(\sfp,\sfp)} \int_{\cC^d} \I\{\varrho_x \varrho K\in\,\cdot\,\}\, \BQ(\dint K) \, \kappa(\sfp,\dint \varrho)\\
            &= \int_{\cC^d} \I\{\varrho_x K\in\,\cdot\,\}\, \BQ(\dint K),
    \end{align*}
    where the last inequality follows since \(\BQ\) is invariant under isometries fixing \(\sfp\) and \(\kappa(\sfp,\,\cdot\,)\) is a probability measure.
    Similarly, the representation of the intensity measure \(\E \Psi\) simplifies to
    $$
    (\E\Psi)(\,\cdot\,)=\gamma\int_{\BH^d}\int_{\cC^d}\I\{(x,\varrho_x K)\in\,\cdot\,\}\,\BQ(\dint K)\, \mathcal{H}^d(\dint x)
    $$
and therefore
$$
(\E\eta)(\,\cdot\,)=\gamma\int_{\BH^d}\int_{\cC^d}\I\{\varrho_x K\in\,\cdot\,\}\,\BQ(\dint K)\, \mathcal{H}^d(\dint x).
$$

\section{Intersection density of a hyperbolic Poisson particle process}

We adopt the set-up of the previous section and let $\eta$ be a stationary Poisson particle process on $\cC^d$ with positive intensity and intensity measure given by \eqref{eqdisintegration}. We assume that the probability measure $\BQ$, the distribution of the typical grain of $\eta$, is concentrated on the subset of $\cC^d_\sfp$ which only consists of sets whose boundaries are $(d-1)$-rectifiable. We recall that a subset $G\subset\BH^d$ is called $m$-rectifiable, for some $m\in\{0,1,\ldots,d\}$, if $G$ is the Lipschitz image of some bounded subset of $\mathbb{R}^m$. In particular, each compact $m$-dimensional $C^1$-submanifold of $\BH^d$ is $m$-rectifiable. That this class of sets is a measurable subset of $\cC^d$ follows from the results of \cite[Section 2]{Zaehle82} as explained in the remark on page 184 of \cite{Zaehle83}.

We write $s_d(\eta,B)$ for the number of intersection points of the boundaries of any $d$ distinct particles of $\eta$ within a bounded measurable set $B\subset\BHd$, that is,
$$
s_d(\eta,B):=\frac{1}{d!}\sum_{(K_1,\ldots,K_d)\in\eta^d_{\neq}}\mathcal{H}^0(B\cap \partial K_1\cap\ldots\cap\partial K_d),
$$
where $\eta^d_{\neq}$ denotes the $d$-th factorial measure of $\eta$ (see \cite[Section 4.2]{LP18}). In Theorem \ref{thm:ID} we show that $s_d(\eta,B)<\infty$ holds almost surely. 
Since $\eta$ is stationary, $\E[s_d(\eta,\,\cdot\,)]$ defines an isometry invariant measure on Borel sets in $\BHd$. If this measure is locally finite, it is a multiple of the $d$-dimensional Hausdorff measure. The following theorem provides more precise information and shows that the multiplicity can be expressed in terms of the mean $(d-1)$-dimensional Hausdorff measure of the boundary of the typical grain   and the intensity $\gamma$ of $\eta$. Recall that $\kappa_d=\pi^{d/2}/\Gamma(1+d/2)$, $d\in\N_0$, and $\omega_d=d\kappa_d$ for $d\in\N$.

\begin{theorem}\label{thm:ID}
Let $\eta$ be a stationary Poisson particle process in $\cC^d$ whose particles almost surely have $(d-1)$-rectifiable boundaries, with positive intensity $\gamma$  and distribution $\BQ$ of the typical grain.  If $B\subset\BHd$ is a bounded Borel set, then $s_d(\eta,B)<\infty$ holds  almost surely and 
    $$
    \E[s_d(\eta,B)]=\kappa_d \left(\frac{\kappa_{d-1}}{d\kappa_d}\cdot\gamma \int_{\cC^d_{\sfp}}\mathcal{H}^{d-1}(\partial G)\, \BQ(\dint G)\right)^d\mathcal{H}^d(B).
    $$
\end{theorem}

The factor in front of $\mathcal{H}^d(B)$ is called the \textit{intersection density} $\gamma_d(\eta)$ of $\eta$, that is,
$$
\gamma_d(\eta)=\kappa_d \left(\frac{\kappa_{d-1}}{\omega_d}\cdot\gamma\int_{\cC^d_{\sfp}}\mathcal{H}^{d-1}(\partial G)\, \BQ(\dint G)\right)^d\in [0,\infty].
$$
If $\eta$ has almost surely convex grains, then $\gamma_d(\eta)<\infty$, since $\Lambda$ is locally finite. In fact, this follows from \cite[Theorem 2.3]{HLS2} and the Steiner formula. 

\begin{remark}
{\rm Theorem \ref{thm:ID} holds in all space forms (with proper definitions of the required concepts). The Euclidean case is analyzed in \cite[Section 4.6]{SW08}, especially on pages 153--154, also for anisotropic particle processes, but the analysis requires the restriction to convex grains.
}    
\end{remark}

Our proof of Theorem \ref{thm:ID} is based on a general version of Poincar\'e's formula in hyperbolic space, which is provided in \cite{Brothers66}; see Theorem 5.15 in combination with Definition 2.13 and Remark 8.16. For smooth compact submanifolds, versions and extensions of this result are contained in \cite{Howard}; see Theorem 3.8 and Corollary 3.9. We will need the following special case, which we state as a lemma. If $k\in\{0,d\}$ or $l\in\{0,d\}$, then the assertion of the lemma follows from Fubini's theorem and \eqref{eq:normalize}. 

\begin{lemma}\label{Le:Brothers}
If $A\subset\BHd$ is a Hausdorff $k$-rectifiable Borel set and $C\subset\BHd$ is an $l$-rectifiable Borel set with $k,l\in\{0,\ldots,d\}$ and $k+l\ge d$, then
$$
\int_{\Ih}\mathcal{H}^{k+l-d}(A\cap \varrho C)\, \lambda(\dint\varrho)=\gamma(d,k,l)\mathcal{H}^k(A)\mathcal{H}^l(C),
$$
where
$$
\gamma(d,k,l)=\frac{k!\kappa_k l!\kappa_l}{d!\kappa_d(k+l-d)!\kappa_{k+l-d}}.
$$
Moreover, for $\lambda$-almost all $\varrho\in \Ih$ the intersection $A\cap\varrho C$ is a Hausdorff $(k+l-d)$-rectifiable Borel set.
\end{lemma}

\begin{proof}
The second statement follows from \cite[Theorem 4.4 (4) (iii)]{Brothers66} (see also the statement about transversal intersections in \cite[Theorem 3.8]{Howard}) and the proof of Theorem 5.15 in \cite{Brothers66}. 
\end{proof}

We need the following iterated version.

\begin{lemma}\label{Le:Brothers2}
Let $i\in\{1,\ldots,d\}$. If $B\subset\BHd$ is a bounded Borel set and  $C_1,\ldots,C_i\subset\BHd$ are $(d-1)$-rectifiable Borel sets, then 
\begin{align}\label{eq:integralgeom}
&\int_{(\Ih)^i}\mathcal{H}^{d-i}(B\cap \varrho_1 C_1\cap \ldots\cap \varrho_i C_i)\, \lambda^i(\dint(\varrho_1,\ldots,\varrho_i))\nonumber\\
&\quad 
=\frac{d!\kappa_d}{(d-i)!\kappa_{d-i}}\left(\frac{\kappa_{d-1}}{d\kappa_d}\right)^i\mathcal{H}^d(B)\mathcal{H}^{d-1}(C_1)\cdots \mathcal{H}^{d-1}(C_i).
\end{align}
Moreover, for $\lambda^i$-almost all $(\varrho_1,\ldots,\varrho_i)\in (\Ih)^i$ the intersection $B\cap \varrho_1 C_1\cap \ldots\cap \varrho_i C_i$ is a Hausdorff $(d-i)$-rectifiable Borel set. 
\end{lemma}

\begin{proof} Note that $B$ is $d$-rectifiable. 
    We proceed by induction over $i\in\{1,\ldots,d\}$. For $i=1$ the assertion is the special case $(k,l)=(d,d-1)$ with $\gamma(d,d,d-1)=1$ of  Lemma \ref{Le:Brothers}. Assume that the assertion has been proved for $i\in \{1,\ldots,d-1\}$. Then $B\cap \varrho_1 C_1\cap \ldots\cap \varrho_{i-1} C_{i-1}$ is a Hausdorff $(d-i+1)$-rectifiable Borel set for $\lambda^{i-1}$-almost all $(\varrho_1,\ldots,\varrho_{i-1})\in (\Ih)^{i-1}$. Since $C_i$ is a $(d-1)$-rectifiable Borel set, the rectifiability statement for $B\cap \varrho_1 C_1\cap \ldots\cap \varrho_i C_i$ follows from the rectifiability statement of Lemma \ref{Le:Brothers} and Fubini's theorem. 
    
    To verify the integral geometric formula,  we  apply for $\lambda^{i-1}$-almost all $(\varrho_1,\ldots,\varrho_{i-1})\in (\Ih)^{i-1}$  Lemma \ref{Le:Brothers}  with the Hausdorff $(d-i+1)$-rectifiable Borel set $B\cap \varrho_1 C_1\cap \ldots\cap \varrho_{i-1} C_{i-1}$ and the $(d-1)$-rectifiable Borel set $C_i$. By means of Fubini's theorem we obtain
    \begin{align*}
        &\int_{(\Ih)^i}\mathcal{H}^{d-i}(B\cap \varrho_1 C_1\cap \ldots\cap \varrho_i C_i)\, \lambda^i(\dint(\varrho_1,\ldots,\varrho_i))\\
        &\quad=\int_{(\Ih)^{i-1}}\int_{\Ih}\mathcal{H}^{d-i}([B\cap \varrho_1 C_1\cap \ldots \varrho_{i-1}C_{i-1}]\cap \varrho_i C_i)\,\lambda(\dint\varrho_i)\, \lambda^{i-1}(\dint(\varrho_1,\ldots,\varrho_{i-1}))\\
        &\quad=\gamma(d,d-i+1,d-1)\\
        &\qquad \times 
        \int_{(\Ih)^{i-1}}\mathcal{H}^{d-i+1}(B\cap \varrho_1 C_1\cap \ldots\cap \varrho_{i-1} C_{i-1})\, \lambda^{i-1}(\dint(\varrho_1,\ldots,\varrho_{i-1})) \, \mathcal{H}^{d-1}(C_i)\\
        &\quad=\gamma(d,d-i+1,d-1)\frac{d!\kappa_d}{(d-i+1)!\kappa_{d-i+1}}\left(\frac{\kappa_{d-1}}{d\kappa_d}\right)^{i-1}\,\mathcal{H}^d(B)\mathcal{H}^{d-1}(C_1)\cdots \mathcal{H}^{d-1}(C_i),
    \end{align*}
    where in the last step the induction hypothesis has been used. The assertion follows after simplification of the constants.    
\end{proof}

\begin{remark}{\rm 
    Relation \eqref{eq:integralgeom} remains true for an arbitrary Borel set $B\subseteq\BH^d$ and for sets $C_1,\ldots,C_i\subset\BH^d$ such that $C_i\cap W$ is a $(d-1)$-rectifiable Borel set whenever $W\subset\BH^d$ is a bounded Borel set. This follows from Lemma \ref{Le:Brothers2} via the monotone convergence theorem.}
\end{remark}

\begin{proof}[Proof of Theorem \ref{thm:ID}]
    Starting from the definition of $s_d(\eta,B)$, we use the multivariate Mecke formula, \cite[Theorem 4.4]{LP18} and finally Lemma \ref{Le:Brothers2}. Using also \eqref{eqdisintegration},  we obtain 
    \begin{align*}
&\E[s_d(\eta,B)]=\E\left[\frac{1}{d!}\sum_{(K_1,\ldots,K_d)\in\eta^d_{\neq}}\mathcal{H}^0(B\cap\partial K_1\cap\ldots\cap\partial K_d)\right] \\
&\quad=\frac{\gamma^d}{d!}\int_{(\cC^d_\sfp)^d}\int_{(\Ih)^d}
\mathcal{H}^0(B\cap \varrho_1\partial G_1\cap \ldots\cap \varrho_d\partial G_d)\, \lambda^d(\dint(\varrho_1,\ldots,\varrho_d))\, \BQ^d(\dint(G_1,\ldots,G_d))\\
&\quad=\frac{\gamma^d}{d!}\int_{(\cC^d_\sfp)^d}
d!\kappa_d\left(\frac{\kappa_{d-1}}{d\kappa_d}\right)^d\mathcal{H}^d(B)\mathcal{H}^{d-1}(\partial G_1)\cdots \mathcal{H}^{d-1}(\partial G_d)
\, \BQ^d(\dint(G_1,\ldots,G_d))\\
&\quad=\kappa_d\left(\frac{\kappa_{d-1}}{d\kappa_d}\right)^d
\mathcal{H}^d(B)\left(\gamma\int_{\cC^d_\sfp}\mathcal{H}^{d-1}(\partial G)\, \BQ(\dint G)\right)^d,
    \end{align*}
by another application of Fubini's theorem in the last step. 

In order to verify that $s_d(\eta,B)<\infty$ holds almost surely, we consider
 \begin{align}\label{eq:5infinity}
&\E\left[\sum_{(K_1,\ldots,K_d)\in\eta^d_{\neq}}\I\left\{\mathcal{H}^0(B\cap\partial K_1\cap\ldots\cap\partial K_d)=\infty\right\}\right] \\
&\quad=\gamma^d\int_{(\cC^d_\sfp)^d}\int_{(\Ih)^d}
\I\left\{\mathcal{H}^0(B\cap \varrho_1\partial G_1\cap \ldots\cap \varrho_d\partial G_d)=\infty\right\}\, \nonumber \\
&\hspace{6cm}  \lambda^d(\dint(\varrho_1,\ldots,\varrho_d))\,\BQ^d(\dint(G_1,\ldots,G_d)).\nonumber
\end{align}
For given $(d-1)$-rectifiable sets $\partial G_1,\ldots,\partial G_d$, Lemma \ref{Le:Brothers2} implies that
$$
\mathcal{H}^0(B\cap \varrho_1\partial G_1\cap \ldots\cap \varrho_d\partial G_d)<\infty\quad \text{for } \lambda^d\text{-a.a. } (\varrho_1,\ldots,\varrho_d)\in(\Ih)^d, 
$$
since the right-hand side of \eqref{eq:integralgeom} is finite in this situation. Hence we conclude that for $\BQ^d$-a.e. $(G_1,\ldots,G_d)\in (\cC^d_\sfp)^d$ the inner integral on the right-hand side of \eqref{eq:5infinity} is zero. This shows that almost surely
$$
\sum_{(K_1,\ldots,K_d)\in\eta^d_{\neq}}\I\left\{\mathcal{H}^0(B\cap\partial K_1\cap\ldots\cap\partial K_d)=\infty\right\}=0,
$$
i.e.\ almost surely $\mathcal{H}^0(B\cap\partial K_1\cap\ldots\cap\partial K_d)<\infty$ for $(K_1,\ldots,K_d)\in\eta^d_{\neq}$. Since $\eta$ is locally finite, almost surely we also have $B\cap\partial K_1\cap\ldots\cap\partial K_d\neq \varnothing$ for at most finitely many $(K_1,\ldots,K_d)\in\eta^d_{\neq}$. Thus we conclude that $s_d(\eta,B)<\infty$ holds almost surely.
\end{proof}

\section{Mean visible volume in a hyperbolic Boolean model}

In this section, we assume that $\eta$ is a stationary Poisson particle process on $\cC^d$, concentrated on $\cK^d$. The associated union set 
$$
Z:=\bigcup_{K\in\eta}K
$$
is called a stationary \textit{Boolean model}. As in the previous section, we write $\BQ$ for the distribution of the typical grain. For $x,y\in\BHd$,  let $[x,y]$ denote the geodesic segment connecting $x$ and $y$.  The visibility from $\sfp$ in direction $u\in \bS_{\sfp}$, the unit sphere in the tangent space of $\BHd$ at $\sfp$, within $\BHd\,\setminus Z$ is described by the random variable 
\begin{equation}\label{eq:VisRange}
s_{\sfp,u}(Z):=\sup\{s\ge 0:[\sfp,\exp_\sfp(su)]\cap Z=\varnothing\},    
\end{equation}
where $\exp_\sfp$ stands for the exponential map at $\sfp$, which maps a point $su$, with $s\geq 0$  and $u\in\bS_\sfp$, in the tangent space of $\BH^d$ at $\sfp$ to the point in $\BH^d$ at hyperbolic distance $s$ on the geodesic through $\sfp$ in direction $u$.
We call $s_{\sfp,u}(Z)$ the \textit{visibility range} from $\sfp$ in direction $u$. 
The \textit{mean visible volume} from $\sfp$ outside $Z$ is the quantity defined by the (elementary) conditional expectation value 
$$
\overline{\vol}_s(Z):=\E[\mathcal{H}^d(S_\sfp(Z))\mid p\notin Z],
$$
where the open and star-shaped set
$$
S_\sfp(Z):=\{y\in\BHd:[\sfp,y]\cap Z=\varnothing\},
$$
is the \textit{visible region} seen from $\sfp$. For an introduction to these quantities in Euclidean space (and the proof of basic measurability statements) we refer to \cite[Section 4.6, pp.~150--153]{SW08}. We aim at determining the mean visible volume $\overline{\vol}_s(Z)$ of a Boolean model in hyperbolic space. As a first step, we consider the conditional distribution of $s_{\sfp,u}(Z)$ given that $\sfp\notin Z$. For a set $A\subset\BH^d$, we denote -- as usual -- by $\mathcal{F}_A$ the collection of all closed subsets of $\BH^d$ having nonempty intersection with $A$. 
Using this notation, let $r\ge 0$ and observe that
\begin{align}\label{eq:chi1}
&\BP(s_{\sfp,u}(Z)\le r\mid \sfp\notin Z)
=1-\BP(\sfp\notin Z)^{-1}\cdot \BP\left(Z\cap [\sfp,\exp_\sfp(ru)]=\varnothing\right)\nonumber\\   
&\quad=1-\BP(\sfp\notin Z)^{-1}\cdot \BP\left(\eta\left(\mathcal{F}_{[\sfp,\exp_\sfp(ru)]}\right)=0\right)\nonumber\\
&\quad=1-\BP(\sfp\notin Z)^{-1}\cdot\exp\left(-\Lambda\left(\mathcal{F}_{[\sfp,\exp_\sfp(ru)]}\right)\right)\nonumber\\
&\quad=1-\BP(\sfp\notin Z)^{-1}\cdot\exp\left(-\gamma
\int_{\cK^d_\sfp}\int_{\Ih}\chi\left( [\sfp,\exp_\sfp(ru)]\cap \varrho G\right) \, \lambda(\dint\varrho)\, \BQ(\dint G)\right).
\end{align}
In order to determine the integral in the last expression, we apply the principal kinematic formula for the Euler characteristic $\chi$ in hyperbolic space. As in \cite[Section 3]{HLS2}, for $d\ge 2$ we use  the normalized functionals 
\begin{align}\label{eq:vk0}
V_k^0(A):= \frac{\omega_{d+1}}{\omega_{k+1}\omega_{d-k}}\,\binom{d-1}{k} V_k(A)=\frac{\omega_{d+1}}{\omega_{k+1}} V_k^{\mathbb{H}^d}(A),\quad k\in\{0,\ldots,d-1\},
\end{align}
and set $V_d^0(A):= V_d(A)=V_d^{\mathbb{H}^d}(A)$ for $A\in\cK^d$. Note that 
\begin{equation}\label{eq:basicrelations}
V^0_1(A)=V_1(A),\qquad V_{d-1}^0(A)=\frac{\omega_{d+1}}{2\omega_d}V_{d-1}(A)\qquad\text{and}\qquad V_{0}^0(A)=\frac{\omega_{d+1}}{2\omega_d}V_{0}(A).
\end{equation}
In particular, for $d=2$ we get $V^0_i=V_i$ for all possible choices of $i\in\{0,1,2\}$. In arbitrary dimension $d\ge 1$ and for $A\in\mathcal{R}^d$ (the set of all finite unions of compact convex sets), the relation between the Euler characteristic and the intrinsic volumes can be expressed  in the form
\begin{equation}\label{eq:euler}
\chi(A)=\frac{2}{\omega_{d+1}}\sum_{l= 0}^{\lfloor \frac{d}{2}\rfloor} (-1)^l\, V_{2l}^0(A);
\end{equation}
see \cite[Equation (3.11)]{HLS2}, \cite[Th\'eor\`eme 1.2]{BeBoe2003}, and the references provided there. (For $d=1$, this follows from $\chi=\frac{1}{2}V_0$ together with $V^0_0\defeq\frac{\pi}{2}V_0$.) 
In terms of the modified functionals $V_k^0$, the principal kinematic formulas turn into  
 \begin{align}\label{eq:kinematic}
\int V_k^0(A\cap \varrho B)\, \lambda(\dint\varrho)=\sum_{i+j=d+k}V_i^0(A)V_j^0(B),\qquad k\in\{0,\ldots,d\},
\end{align}
where the summation extends over all $i,j\in \{0,\ldots,d\}$ such that $i+j=d+k$; see \cite[Equation (3.11)]{HLS2}, \cite[Th\'eor\`eme 4.4]{BeBoe2003}, and the references provided there. In the following, we use that $V_j^0(A)=0$ if $j>\dim A$, which is implied by Proposition \ref{prop:intrinsic_hyp}. 

Next, we consider the inner integral in \eqref{eq:chi1}, first apply \eqref{eq:euler} and then \eqref{eq:kinematic}. Using also Proposition \ref{prop:intrinsic_hyp} and \eqref{eq:basicrelations}, we get
\begin{align}\label{eq:innerintegral}
    &\int_{\Ih}\chi\left( [\sfp,\exp_\sfp(ru)]\cap \varrho G\right) \, \lambda(\dint\varrho)\nonumber\\
    &\quad=
    \frac{2}{\omega_{d+1}}\sum_{l= 0}^{\lfloor \frac{d}{2}\rfloor} (-1)^l \int_{\Ih} V_{2l}^0([\sfp,\exp_\sfp(ru)]\cap \varrho G)\, \lambda(\dint\varrho)\nonumber\\
    &\quad=\frac{2}{\omega_{d+1}}\sum_{l= 0}^{\lfloor \frac{d}{2}\rfloor} (-1)^l\sum_{i+j=d+2l}V_i^0(G)V_j^0([\sfp,\exp_\sfp(ru)])\nonumber\\
    &\quad=\frac{2}{\omega_{d+1}}\sum_{l= 0}^{\lfloor \frac{d}{2}\rfloor} (-1)^l\left(V_{d+2l}^0(G)V_0^0([\sfp,\exp_\sfp(ru)])+V_{d+2l-1}^0(G)V_1^0([\sfp,\exp_\sfp(ru)])\right)\nonumber\\
    &\quad=\frac{2}{\omega_{d+1}}\left(
    V_{d}^0(G)V_0^0([\sfp,\exp_\sfp(ru)])+
    V_{d-1}^0(G)V_1^0([\sfp,\exp_\sfp(ru)])
    \right)\nonumber\\
    &\quad=\frac{2}{\omega_{d+1}}\left(V_{d}(G)\frac{\omega_{d+1}}{2\omega_d}V_0([\sfp,\exp_\sfp(ru)])+\frac{\omega_{d+1}}{2\omega_d}
    V_{d-1}(G)V_1([\sfp,\exp_\sfp(ru)])
    \right)\nonumber\\
    &\quad=\frac{1}{\omega_{d}}\left(V_{d}(G)V_0([\sfp,\exp_\sfp(ru)])+
    V_{d-1}(G)V_1([\sfp,\exp_\sfp(ru)])
    \right)\nonumber\\
    &\quad=\frac{1}{\omega_{d}}\left(\omega_dV_{d}(G) +
    \kappa_{d-1}rV_{d-1}(G)
    \right)\nonumber\\
    &\quad=V_{d}(G)+\frac{\kappa_{d-1}}{d\kappa_d}rV_{d-1}(G),
\end{align}
where we used that $[\sfp,\exp_\sfp(ru)]$ is a one-dimensional geodesic segment, so that  Remark \ref{rem:propintrinsic} implies $V_0([\sfp,\exp_\sfp(ru)])=\omega_d$ and 
$$
V_1([\sfp,\exp_\sfp(ru)])=\kappa_{d-1}\mathcal{H}^1([\sfp,\exp_\sfp(ru)])=\kappa_{d-1}r.
$$
Next, we plug \eqref{eq:innerintegral} back into \eqref{eq:chi1} to deduce that
\begin{align*} 
\BP(s_{\sfp,u}(Z)\le r\mid \sfp\notin Z)
&=1-\BP(\sfp\notin Z)^{-1}\exp\left(- \gamma 
\int_{\cK^d_\sfp}
 \left( V_{d}(G) +
    \frac{\kappa_{d-1}}{d\kappa_d}rV_{d-1}(G)
    \right)
\, \BQ(\dint G)\right).
\end{align*}
Setting $r=0$ in the preceding equation, we conclude that
$$
\BP(\sfp\notin Z)=\BP\left(Z\cap [\sfp,\exp_\sfp(0\cdot u)]=\varnothing\right)=\exp\left(-\gamma 
\int_{\cK^d_\sfp}
   V_{d}(G) 
\, \BQ(\dint G)\right),
$$
and hence we finally obtain
$$
\BP(s_{\sfp,u}(Z)\le r\mid \sfp\notin Z)=1-
\exp\left(-r \frac{\kappa_{d-1}}{d\kappa_d}\gamma
\int_{\cK^d_\sfp}
  V_{d-1}(G)
\, \BQ(\dint G)\right).
$$
Defining
\begin{equation}
\label{eq:def_v_d_minus_1}
v_{d-1} := \int_{\cK^d_\sfp} V_{d-1}(G)\, \BQ(\dint G) \in [0,\infty)
\qquad\text{and}\qquad
v_{d-1}^\ast:=\frac{\kappa_{d-1}}{d\kappa_d} v_{d-1},
\end{equation}
{we have thus} proved the following result.

\begin{proposition}\label{prop:exp}
   Consider a Boolean model with compact convex grains and positive intensity $\gamma$. The conditional distribution of the visible range $s_{\sfp,u}(Z)$, given $\sfp\notin Z$, is an exponential distribution with parameter $v_{d-1}^\ast\gamma$, independently of $u\in\bS_\sfp$, where \(v_{d-1}^*\) is given by \eqref{eq:def_v_d_minus_1}.
\end{proposition}

Next we consider the mean visible volume.

\begin{table}[t]
    \centering
    \small
    \setlength{\tabcolsep}{12pt}
    \renewcommand{\arraystretch}{1.8}

    \begin{tabular}{lll}
        \toprule
        $d=2$ & $d=3$ & $d=4$ \\
        \midrule
        $\displaystyle\overline{\vol}_s(Z)=\frac{2\pi^3}{\gamma^2v_1^2-\pi^2}$ &
        $\displaystyle\overline{\vol}_s(Z)=\frac{512\pi}{\gamma v_2(\gamma^2 v_2^2-64)}$ &
        $\displaystyle\overline{\vol}_s(Z)=\frac{972\pi^6}{(9\pi^2-4\gamma^2 v_3^2)(81\pi^2-4\gamma^2 v_3^2)}$ \\
        \addlinespace
        \text{if }$\gamma v_1>\pi$ &
        \text{if }$\gamma v_2>8$ &
        \text{if }$\gamma v_3>9\pi/2$ \\
        \bottomrule
    \end{tabular}
    \caption{\small Mean visible volume in small space dimensions.}
    \label{tab:d=234}
\end{table}


\begin{theorem}\label{thm:VisVol}
    The mean visible volume $\overline{\vol}_s(Z)$ of a stationary Boolean model $Z$ with compact convex grains and positive intensity $\gamma$ is finite 
    if and only if 
$$v_{d-1}^\ast\gamma> d-1,\quad\text{or equivalently,}\quad v_{d-1}\gamma>(d-1){d\kappa_d\over\kappa_{d-1}}.$$
In this case, 
$$
\overline{\vol}_s(Z)={d!\kappa_d  \over 2^d}{\Gamma\big({v_{d-1}^\ast \gamma-d+1\over 2}\big)\over\Gamma\big({v_{d-1}^\ast \gamma+d+1\over 2}\big)}.
$$
\end{theorem}
\begin{proof}
Using Lemma \ref{lem:IntegrationFormulas} (b), Proposition \ref{prop:exp},  and by an application of Fubini's theorem, we deduce that
\begin{align*}
    \overline{\vol}_s(Z)&=\E[\mathcal{H}^d(S_\sfp(Z))\mid p\notin Z]\\
    &=\E \left[\int_{\bS_\sfp} \int_0^{s_{\sfp,u}(Z)}\sinh^{d-1}(s)\, \dint s\,\mathcal{H}^{d-1}_\sfp (\dint u)\mid \sfp\notin Z\right]\\
    &=\gamma v_{d-1}^\ast \int_{\bS_\sfp} \int_0^\infty \int_{0}^t \sinh^{d-1}(s)e^{-v_{d-1}^\ast \gamma\,  t}\, \dint s\,\dint t \,\mathcal{H}^{d-1}_\sfp (\dint u)\\
    &=\gamma v_{d-1}^\ast \omega_d\int_0^\infty \int_{0}^t \sinh^{d-1}(s)e^{-v_{d-1}^\ast \gamma\,  t}\, \dint s\,  \dint t\\
    &=\gamma v_{d-1}^\ast \omega_d\int_0^\infty \int_{s}^\infty e^{-v_{d-1}^\ast \gamma\,  t}\, \dint t\,  \sinh^{d-1}(s)\, \dint s\\
    &=\omega_d\int_0^\infty e^{-v_{d-1}^\ast \gamma\,  s} \sinh^{d-1}(s)\, \dint s.
\end{align*}
Note that the application of Fubini's theorem was justified, since the integrand is non-negative.

To evaluate this integral, let $a>d-1$ and apply the substitution $e^{-s}=\sqrt{u}$ to see that
\begin{align*}
    \int_0^\infty \sinh^{d-1}(s)\,e^{-as}\,\dint s &= {1\over 2^{d-1}}\int_0^\infty (e^s-e^{-s})^{d-1}\,e^{-as}\,\dint s\\
    &={1\over 2^{d-1}}\int_0^1\left({1\over\sqrt{u}}-\sqrt{u}\right)^{d-1}(\sqrt{u})^{a-1}\,{\dint u\over 2\sqrt{u}}\\
    &={1\over 2^d}\int_0^1 u^{{a-d+1\over 2}-1}(1-u)^{d-1}\,\dint u.
\end{align*}
The last integral can be expressed in terms of a beta function, and hence by gamma functions (as in the proof of \Cref{prop:ell_identity}), which finally implies that
\begin{align}\label{eq:OberwolfachIntegral}
    \int_0^\infty \sinh^{d-1}(s)\,e^{-as}\,\dint s = {(d-1)!\over 2^d}{\Gamma\left({a-d+1\over 2}\right)\over\Gamma\left({a+d+1\over 2}\right)},
\end{align}
provided that $a>d-1$. 
Using this identity with $a=v_{d-1}^\ast \gamma$ yields the final result.
\end{proof}

\begin{remark}\rm 
    Writing the result of Theorem \ref{thm:VisVol} in the form 
    $$
    \overline{\vol}_s(Z)={d!\kappa_d  \over 2^d}{\Gamma\big({v_{d-1}^\ast \gamma-d+1\over 2}\big)\over\Gamma\big({v_{d-1}^\ast \gamma-d+1\over 2}+d\big)},
    $$
    we obtain
$$
 \frac{\overline{\vol}_s(Z)}{d!\kappa_d}=  \begin{cases}
\left(\left((\gamma v_{d-1}^\ast)^2-(d-1)^2\right)\left((\gamma v_{d-1}^\ast)^2-(d-3)^2\right)\cdots \left((\gamma v_{d-1}^\ast)^2-1^2\right)\right)^{-1} &: d\text{ even,}\\
\left(\left((\gamma v_{d-1}^\ast)^2-(d-1)^2\right)\left((\gamma v_{d-1}^\ast)^2-(d-3)^2\right)\cdots \left((\gamma v_{d-1}^\ast)^2-2^2\right)\gamma v_{d-1}^\ast\right)^{-1} &: d\text{ odd.}
 \end{cases}
$$
For small vales of $d$, this is illustrated in Table \ref{tab:d=234}.
\end{remark}

\begin{remark}
{\rm Note that $V_{d-1}=2V_{d-1}^{\BH^d}$ by Remark \ref{rem:propintrinsic} , where the functional $V_{d-1}^{\BH^d}$ has the advantage of being normalized in an ``intrinsic way''. Thus, replacing correspondingly $v_{d-1}$ by $2v_{d-1}^{\BH^d}$, we obtain that the mean visible volume $\overline{\vol}_s(Z)$ is finite if and only if 
$$\gamma v_{d-1}^{\BH^d}>(d-1)\frac{d\kappa_d}{2\kappa_{d-1}}.$$  
}
\end{remark}

\begin{figure}[t]
    \centering
    \includegraphics[width=0.3\linewidth]{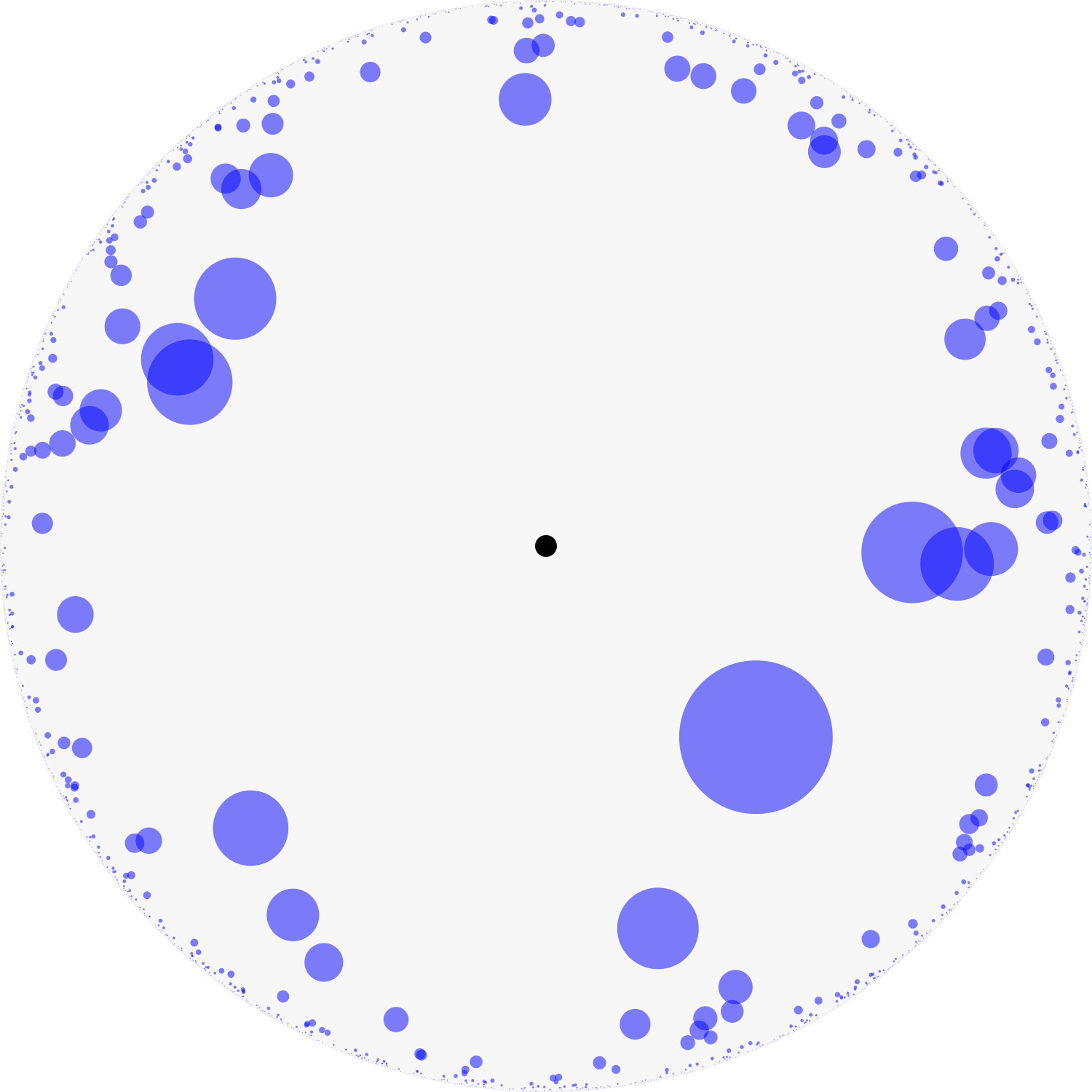}\quad
    \includegraphics[width=0.3\linewidth]{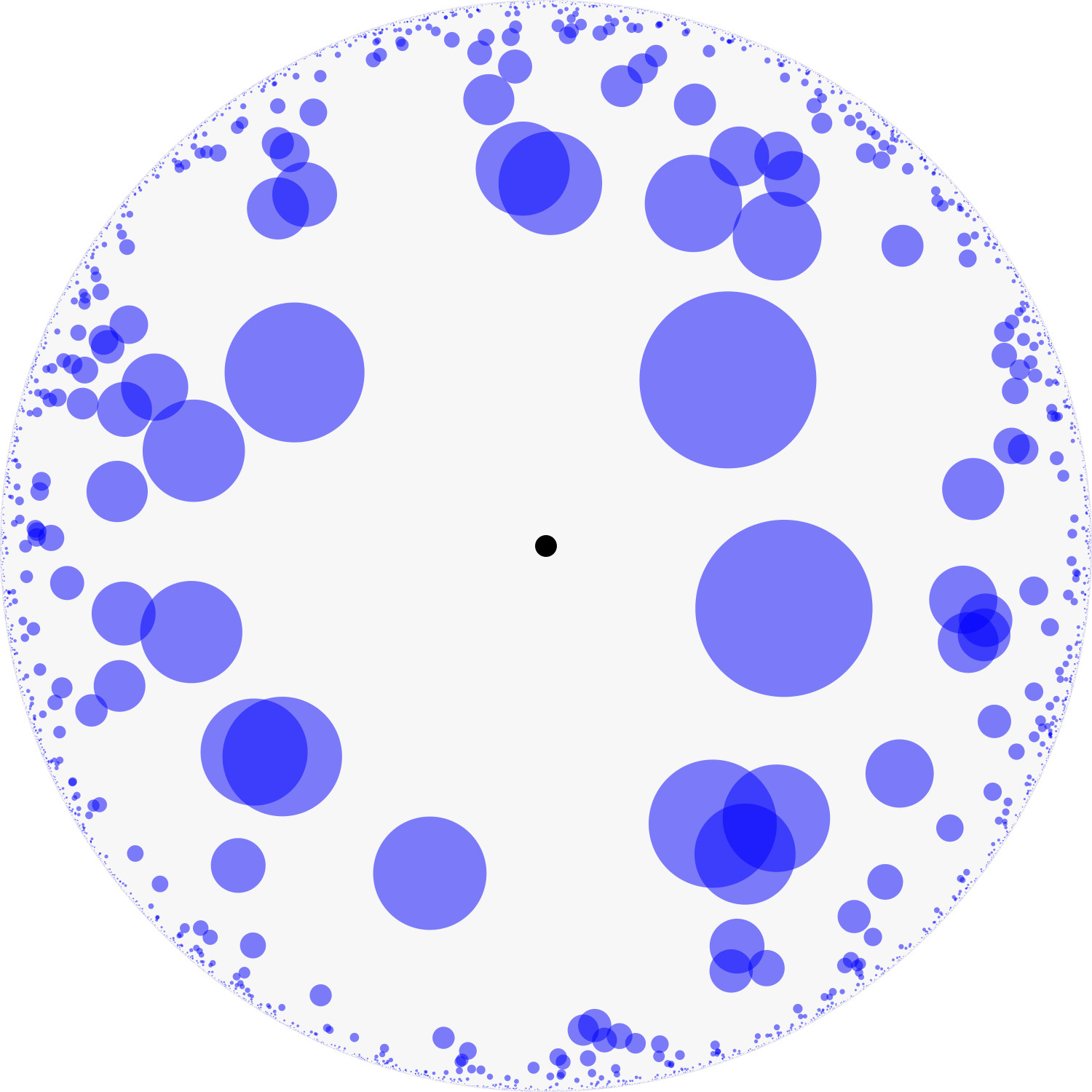}\quad
    \includegraphics[width=0.3\linewidth]{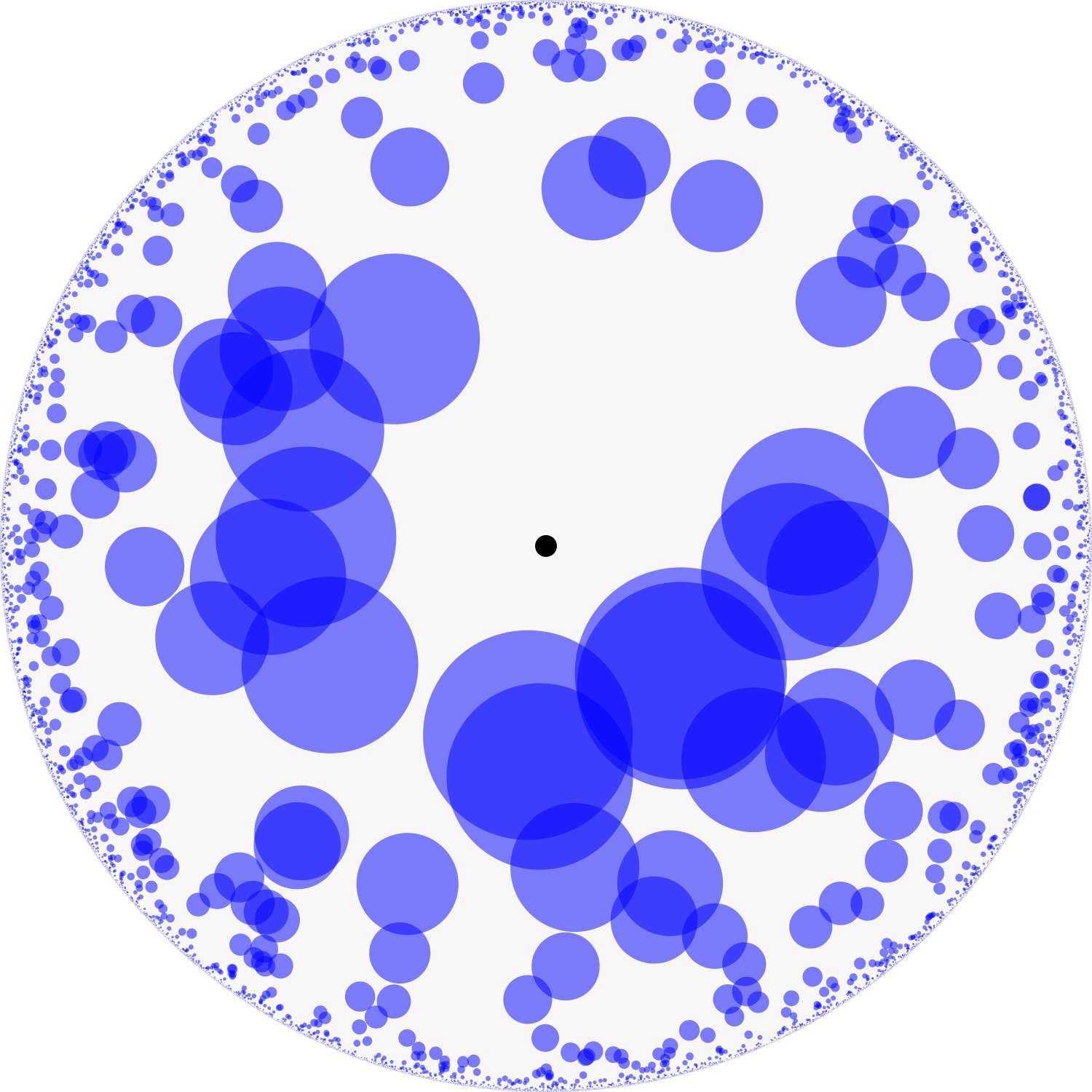}
    \caption{\small Simulations of Boolean models with intensity $0.5$ (left), $0.9595$ (middle) and $1.5$ (right) in the hyperbolic plane whose grains are discs of radius $1/2$. The saturation of a point represents the number of grains by which it is covered. The point $\sfp$ is located at the center.}
    \label{fig:BM2}
\end{figure}

\begin{remark}\label{rem:Visibility}
{\rm
Consider a Boolean model $Z$ generated by a stationary Poisson point process $\xi$ of intensity $\gamma>0$ in $\BH^d$ whose grains are balls with a fixed radius $R>0$, recall the construction from Section \ref{sec:ParticleProcesses}. In this special case, the threshold of Theorem \ref{thm:VisVol} may be compared to the one of \cite{Lyons} considering the absence or presence of infinite geodesic rays starting at $\sfp$ not hitting the Boolean model as discussed in the introduction. It has been shown on page 447 of \cite{Lyons} that such geodesic rays  with positive probability if the intensity $\gamma$  satisfies 
    $$
\gamma<{d-1\over\kappa_{d-1}\sinh^{d-1}(R)}=:\beta_c(d,R),
    $$
whereas for $\gamma>\beta_c(d,R)$ such rays do almost surely not exist;
    in the language of \cite{Lyons} one has to take $a=1$ and to note that $\gamma_1$ there is the same as our $\kappa_{d-1}$.
    On the other hand, Theorem \ref{thm:VisVol} shows that the mean visible volume $\overline{V}_s(Z)$ is infinite if and only if 
    $$
    \gamma\le {(d-1)d\kappa_d\over\kappa_{d-1}v_{d-1}}.
    $$
    Since $v_{d-1}=\omega_d\sinh^{d-1}(R)$ in our case, we arrive at precisely the same threshold $\beta_c(d,R)$. In this very special set-up, this  can alternatively be concluded as follows. Using Lemma \ref{lem:IntegrationFormulas} (b) and writing $x_t$ for a point with $d_h(\sfp,x_t)=t$ we have 
    \begin{align*}
    \overline{\vol}_s(Z) = \omega_d\int_0^\infty \BP([\sfp,x_t]\cap Z=\varnothing\,|\,\sfp\notin Z)\,\sinh^{d-1}t\,\dint t.
    \end{align*}
    Now, observe that
    \begin{align*}
    \BP([\sfp,x_t]\cap Z=\varnothing\,|\,\sfp\notin Z) ={\BP(\xi(\BB([\sfp,x_t],R))=0)\over\BP(\xi(\BB(\sfp,R))=0)}  = {e^{-\gamma V_d(\BB([\sfp,x_t],R))}\over e^{-\gamma V_d(\BB(\sfp,R))}}.
    \end{align*}
    Using \eqref{eq:Steinerhyp} and Remark \ref{rem:propintrinsic} we see that
    $$
    e^{-\gamma V_d(\BB([\sfp,x_t],R))} = e^{-\gamma\big(V_d(\BB(\sfp,R))+t\kappa_{d-1}\sinh^{d-1}(R)\big)} = e^{-\gamma\big(V_d(\BB(\sfp,R))+{(d-1)t\over\beta_c(d,R)}\big)},
    $$
    where in the last equality we applied the definition of $\beta_c(d,R)$.
    Thus,
    $$
    \overline{\vol}_s(Z) = \omega_d\int_0^\infty e^{-\gamma{(d-1)t\over\beta_c(d,R)}}\,\sinh^{d-1}t\,\dint t.
    $$
    Since $\sinh t$ is bounded from above and below by a constant multiple of $e^t$ for sufficiently large $t$, the last integral is finite if and only if $1-{\gamma\over\beta_c(d,R)}<0$, or equivalently, $\gamma>\beta_c(d,R)$. 
}
\end{remark}

In view of Theorem \ref{thm:ID}, we also arrive at the following corollary of Theorem \ref{thm:VisVol}.

\begin{corollary}
The mean visible volume of the Boolean model $Z$ with positive intensity and almost surely full-dimensional compact convex grains is finite if and only if the intersection density of the underlying Poisson particle process $\eta$ satisfies $\gamma_d(\eta)>(d-1)^d\kappa_d $.    
\end{corollary}

{Finally, we consider the behavior of the mean visible volume near the critical threshold. Our next result shows that there is a universal scaling law. In particular, the divergence of $\overline{\vol}_s(Z)$ at criticality is polynomial (rather than exponential), with critical exponent one, independent of the dimension $d$. In what follows, we write $f(x)\sim g(x)$ as $x\to x_0$ if $f(x)/g(x)\to 1$ as $x\to x_0$.

\begin{corollary}\label{cor:AboveCriticality}
Let $Z$ be a stationary Boolean model in $\BHd$ with compact convex grains and intensity $\gamma>0$, and set $a\defeq v_{d-1}^\ast\gamma$. 
Assume that $a>d-1$ and define $\delta\defeq a-(d-1)>0$. Then
\begin{equation}\label{eq:critical_asymptotic}
\overline{\vol}_s(Z)\sim \frac{\omega_d}{2^{d-1}}\cdot\frac{1}{\delta}
\qquad\text{as }\delta\downarrow 0.
\end{equation}
\end{corollary}
\begin{proof}
By Theorem \ref{thm:VisVol} we may write, for $a>d-1$,
\begin{equation}\label{eq:gamma_form_cor}
\overline{\vol}_s(Z)=\frac{d!\kappa_d}{2^d}\,
\frac{\Gamma\big(\frac{a-d+1}{2}\big)}{\Gamma\big(\frac{a+d+1}{2}\big)}.
\end{equation}
Since $\delta\defeq a-(d-1)>0$, we have
\[
\frac{a-d+1}{2}=\frac{\delta}{2}
\qquad\text{and}\qquad
\frac{a+d+1}{2}=d+\frac{\delta}{2}.
\]
Using the asymptotics $\Gamma(\varepsilon)\sim \varepsilon^{-1}$ as $\varepsilon\downarrow 0$ and the continuity of the gamma function at $d$, we obtain
\[
\Gamma\Big(\frac{\delta}{2}\Big)\sim \frac{2}{\delta}
\qquad\text{and}\qquad
\Gamma\Big(d+\frac{\delta}{2}\Big)\to \Gamma(d)=(d-1)!\qquad\text{as }\delta\downarrow 0.
\]
Inserting this into \eqref{eq:gamma_form_cor} yields
\[
\overline{\vol}_s(Z)\sim \frac{d!\kappa_d}{2^d}\cdot \frac{2/\delta}{(d-1)!}
=\frac{d\kappa_d}{2^{d-1}}\cdot\frac{1}{\delta}.
\]
Since $\omega_d=d\kappa_d$, this proves \eqref{eq:critical_asymptotic}.
\end{proof}

A next natural question is to quantify the divergence of the mean visible volume \textit{at} the critical threshold in a more refined way. Since $\overline{\vol}_s(Z)=\infty$ when $v_{d-1}^\ast\gamma=d-1$, it is natural to localize the visible region and to study how the expected visible volume grows with the observation radius. To this end, for $R>0$ we consider the truncated mean visible volume
$$
\overline{\vol}_{s,R}(Z)\defeq \E\big[\mathcal{H}^d\big(S_\sfp(Z)\cap \BB(\sfp,R)\big)\mid \sfp\notin Z\big].
$$
Our next result shows that the divergence is governed by a universal constant, independent of the concrete shapes of the particles, and is of a mild order: at criticality the truncated mean visible volume grows only linearly in $R$. For completeness, we also consider the asymptotics of the mean visible volume below and above criticality, which in these cases is of exponential order and depends on $\gamma v_{d-1}^\ast$.

\begin{corollary}\label{cor:AtCriticality}
Let $Z$ be a stationary Boolean model in $\BHd$ with compact convex grains and intensity $\gamma>0$, and set $a\defeq v_{d-1}^\ast\gamma$. 
Then, as $R\to\infty$, the following asymptotics hold.
\begin{enumerate}
\item[\normalfont{(a)}] If $a<d-1$, then
\begin{equation}\label{eq:subcritical_trunc_asymptotic}
\overline{\vol}_{s,R}(Z)\sim \frac{\omega_d}{2^{d-1}(d-1-a)}\,e^{(d-1-a)R}.
\end{equation}
\item[\normalfont{(b)}] If $a=d-1$, then
\begin{equation}\label{eq:critical_trunc_asymptotic}
\overline{\vol}_{s,R}(Z)\sim \frac{\omega_d}{2^{d-1}}\,R.
\end{equation}
\item[\normalfont{(c)}] If $a>d-1$, then 
\begin{equation}\label{eq:supercritical_trunc_asymptotic}
\overline{\vol}_s(Z)-\overline{\vol}_{s,R}(Z)\sim \frac{\omega_d}{2^{d-1}(a-d+1)}\,e^{-(a-d+1)R}.
\end{equation}
\end{enumerate}
\end{corollary}
\begin{proof}
By Lemma \ref{lem:IntegrationFormulas} (b) and Proposition \ref{prop:exp}, the same computation as in the proof of Theorem \ref{thm:VisVol} yields
\begin{equation}\label{eq:trunc_rep}
\overline{\vol}_{s,R}(Z)=\omega_d\int_0^{R} e^{-as}\sinh^{d-1}(s)\,\dint s.
\end{equation}
Moreover, if $a>d-1$, then Theorem \ref{thm:VisVol} implies that
\begin{equation}\label{eq:tail_rep}
\overline{\vol}_s(Z)-\overline{\vol}_{s,R}(Z)=\omega_d\int_R^{\infty} e^{-as}\sinh^{d-1}(s)\,\dint s.
\end{equation}
Using $\sinh (s)=\frac12(e^s-e^{-s})$,  for all $s\ge 0$ we have
\begin{equation}\label{eq:integrand_factor}
e^{-as}\sinh^{d-1}(s)
=2^{-(d-1)}e^{-(a-d+1)s}\big(1-e^{-2s}\big)^{d-1}.
\end{equation}
In particular,
\begin{equation}\label{eq:integrand_asymptotic}
e^{-as}\sinh^{d-1}(s)\sim 2^{-(d-1)}e^{-(a-d+1)s}
\qquad\text{as }s\to\infty.
\end{equation}

\smallskip
\noindent \normalfont{(a)} Assume $a<d-1$ and set $\beta\defeq d-1-a>0$. Let
\[
I(R)\defeq \int_0^{R} e^{-as}\sinh^{d-1}(s)\,\dint s.
\]
By \eqref{eq:integrand_asymptotic} we have $I(R)\to\infty$ and $I(R)/e^{\beta R}\to {1\over\beta}2^{-(d-1)}$ as $R\to\infty$,
which can be seen by l'Hospital's rule:
\[
\lim_{R\to\infty}\frac{I(R)}{e^{\beta R}}
=\lim_{R\to\infty}\frac{e^{-aR}\sinh^{d-1}(R)}{\beta e^{\beta R}}
=\frac{1}{\beta}\lim_{R\to\infty}\frac{\sinh^{d-1}(R)}{e^{(d-1)R}}
=\frac{1}{\beta}\cdot\frac{1}{2^{d-1}
},
\]
where in the last step we used that $\sinh (R)\sim \frac12 e^R$.
Multiplying by $\omega_d$ and using \eqref{eq:trunc_rep}, we conclude  \eqref{eq:subcritical_trunc_asymptotic}.

\smallskip
\noindent\normalfont{(b)}
Assume $a=d-1$. Then \eqref{eq:integrand_factor} gives
\[
e^{-(d-1)s}\sinh^{d-1}(s)=2^{-(d-1)}\big(1-e^{-2s}\big)^{d-1}.
\]
Inserting this into \eqref{eq:trunc_rep} and using L'Hospital's rule, we obtain
\[
\overline{\vol}_{s,R}(Z)=\frac{\omega_d}{2^{d-1}}\int_0^R \big(1-e^{-2s}\big)^{d-1}\,\dint s \sim \frac{\omega_d}{2^{d-1}} R,
\]
which yields \eqref{eq:critical_trunc_asymptotic}.

\smallskip
\noindent\normalfont{(c)}
Assume $a>d-1$ and set $\beta\defeq a-d+1>0$. If
\[
J(R)\defeq \int_R^{\infty} e^{-as}\sinh^{d-1}(s)\,\dint s,
\]
then $J(R)\downarrow 0$ as $R\to \infty$. Again by L'Hospital's rule and using  \eqref{eq:integrand_asymptotic}, we get $J(R)/e^{-\beta R}\to {1\over\beta}2^{-(d-1)}$ as $R\to\infty$. 
Multiplying by $\omega_d$ and using \eqref{eq:tail_rep}, we obtain \eqref{eq:supercritical_trunc_asymptotic}.
\end{proof}}

\section{Comparison with hyperbolic Poisson hyperplane tessellations}

We consider a Poisson process \(\xi\) of hyperplanes in \(\BHd\) with intensity \(\gamma>0\).
By this we mean that $\xi$ is a Poisson process in the space \(A_h(d,d-1)\) with intensity measure \(\gamma \mu_{d-1}\), where
\[
\mu_{d-1}(\,\cdot\,) = \int_{G_h(d,1)} \int_L \cosh^{d-1}(d_h(x,\sfp)) \I\{H(L,x) \in \,\cdot\,\} \,\mathcal{H}^1(\dint x) \,\nu_1(\dint L),
\]
is the isometry invariant measure on the space of hyperplanes in $\BHd$, normalized as in \cite{BHT,HHT21}.
Here, \(G_h(d,1)\) is the space of lines (that is, $1$-planes) containing \(\sfp\), \(\nu_1\) denotes the unique Borel probability measure on \(G_h(d,1)\) invariant under isometries fixing \(\sfp\) and \(H(L,x)\) is the hyperplane orthogonal to \(L\) at \(x\). We refer to \cite{BHT,HeroldDiss,HeHu+,HHT21} for more details.
In this section, let
\begin{equation}\label{eq:Zpht}
Z^\ast := \bigcup_{H \in \xi} H    
\end{equation}
be the random closed set in $\BHd$ induced by the union of all hyperplanes of $\xi$. Since $\mu_{d-1}$ is isometry invariant, the law of the random set $Z$ is also invariant under all isometries of $\BHd$.
We call the closure of the almost surely unique connected component in the complement of $Z^\ast$ containing $\sfp$, that is, the cell in the tessellation generated by \(\xi\) that contains \(\sfp\), the \textit{zero cell} of the hyperbolic Poisson hyperplane tessellation with intensity $\gamma$. By definition, the zero cell coincides with the closure of the visible region \(S_\sfp(Z^\ast)\) induced by the random closed set $Z^\ast$ defined via \eqref{eq:Zpht}.

Defining the visibility range $s_{\sfp,u}(Z^\ast)$ induced by $Z^\ast$ from $\sfp$ in direction $u\in\mathbb{S}^{d-1}_\sfp$ as in \eqref{eq:VisRange}, we can argue as in \eqref{eq:chi1} to deduce that
\begin{align*}
\BP(s_{\sfp,u}(Z^\ast)\le r)
&=1-\BP\left(Z^\ast\cap [\sfp,\exp_\sfp(ru)]=\varnothing\right)\nonumber\\   
&=1-\BP\left(\xi\left(\mathcal{F}_{[\sfp,\exp_\sfp(ru)]}\right)=0\right)\nonumber\\
&=1-\exp\left(-\gamma \cdot \mu_{d-1}\left(\mathcal{F}_{[\sfp,\exp_\sfp(ru)]}\right)\right)\nonumber\\
&=1-\exp\left(-\gamma \cdot \int_{A_h(d,d-1)}\I\left\{ [\sfp,\exp_\sfp(ru)]\cap H \neq\varnothing\right\} \, \mu_{d-1}(\dint H) \right).
\end{align*}
Next, we use that the geodesic segment $[\sfp,\exp_\sfp(ru)]$ can be intersected in at most one point by a generic hyperplane $H\in A_h(d,d-1)$, which implies that $\I\{ [\sfp,\exp_\sfp(ru)]\cap H \neq\varnothing\}=\mathcal{H}^0( [\sfp,\exp_\sfp(ru)]\cap H )$ for $\mu_{d-1}$-almost all $H\in A_h(d,d-1)$. Thus, by the Crofton formula (see, e.g., \cite[Lemma 2]{HHT21} or \cite[Lemma 2.1]{BHT}), the integral in the exponent is equal to
\[ \int_{A_h(d,d-1)}\mathcal{H}^0\left( [\sfp,\exp_\sfp(ru)]\cap H \right) \, \mu_{d-1}(\dint H)
    = \frac{\omega_{d+1}\omega_{1}}{\omega_{d}\omega_{2}} \mathcal{H}^{1}([p,\exp_p(ru)]) = \frac{2\kappa_{d-1}}{d\kappa_d} r.\]
It follows that the random variable \(s_{\sfp,u}(Z)\) has an exponential distribution with parameter \(\gamma\frac{2\kappa_{d-1}}{d\kappa_d}\).
We can now proceed exactly as in the proof of Theorem \ref{thm:VisVol} to obtain
\begin{align*}
\overline{\vol}_s(Z^\ast) &=\E \int_{\mathbb{S}_\sfp^{d-1}}\int_0^{s_{\sfp,u}(Z^\ast)}\sinh^{d-1}(s)\,\dint s\,\mathcal{H}_{\sfp}^{d-1}(\dint u)\\
&=\omega_d\int_0^\infty\int_0^t\sinh^{d-1}(s)\,\gamma{2\kappa_{d-1}\over d\kappa_d}\exp\Big(-\gamma{2\kappa_{d-1}\over d\kappa_d}\,t\Big)\,\dint s\,\dint t\\
&=2\gamma\kappa_{d-1}\int_0^\infty\int_s^\infty \sinh^{d-1}(s)\exp\Big(-\gamma{2\kappa_{d-1}\over d\kappa_d}\,t\Big)\,\dint t\,\dint s\\
&=\omega_d\int_0^\infty \sinh^{d-1}(s)\,\exp\Big(- \gamma\frac{2\kappa_{d-1}}{d\kappa_d}  s\Big)\,\dint s. 
\end{align*}
The last expression can be further simplified by means of relation \eqref{eq:OberwolfachIntegral}. For example, if \(d=2\), we get $\overline{\vol}_s(Z^\ast)={2\pi^3\over 4\gamma^2-\pi^2}$, which agrees with \cite[Equations (6.4) and (6.6)]{SantaloYanez}, where we note that the authors of that paper use a different normalization for the invariant measure $\mu_{d-1}$. 
In particular, defining
$$
\gamma_d^\ast:=\frac{2 \kappa_{d-1}}{d \kappa_d}\gamma,
$$
we see that the mean volume of the zero cell is finite if and only if
\begin{equation}\label{eq:PHTThreshold}
\gamma_d^\ast >  d-1 .   
\end{equation}
We summarize these findings in the following result.

\begin{theorem}\label{Thm7.1}
Consider the zero cell induced by a stationary Poisson process on $A_h(d,d-1)$ with intensity measure $\gamma\cdot\mu_{d-1}$, $\gamma>0$. Its mean volume $\overline{\vol}_s(Z^\ast)$ is finite if and only if $\gamma_d^\ast>d-1$, and in this case it is given by 
$$
\overline{\vol}_s(Z^\ast) ={d!\kappa_d\over 2^d}{\Gamma\big({\gamma_d^\ast-d+1\over 2}\big)\over\Gamma\big({\gamma_d^\ast+d+1\over 2}\big)}.
$$
\end{theorem}

We note that this result is in line with \cite[Theorem 3.3]{GKT} (and the references cited therein), which states that the random set $S_\sfp(Z^\ast)$ is bounded almost surely if $\gamma_d^\ast$ satisfies \eqref{eq:PHTThreshold}, while $S_\sfp(Z^\ast)$ is unbounded almost surely if $\gamma_d^\ast<d-1$ (note the difference by a factor $2$ compared to \cite{GKT}, which is due to a slightly different normalization of the invariant measure $\mu_{d-1}$). If $d=2$, then \cite[Theorem 3.4]{GKT} even ensures that in the critical case $\gamma_2^\ast=1$ (or equivalently $\gamma=\frac{\pi}{2}$) the random set $S_\sfp(Z^\ast)$ is bounded with probability one; in this situation, the expected value $\overline{\vol}_s(Z^\ast) $ is still infinite, as follows from Theorem \ref{Thm7.1}.

{\begin{remark}\label{rem:zero_cell_critical}\rm
Theorem \ref{Thm7.1} yields the same near-critical scaling as in the Boolean model discussed in Corollaries \ref{cor:AboveCriticality} and \ref{cor:AtCriticality}. Writing $\delta\defeq \gamma_d^\ast-(d-1)$,
we obtain
\[
\overline{\vol}_s(Z^\ast)\sim \frac{\omega_d}{2^{d-1}}\cdot\frac{1}{\delta}
\qquad\text{as }\delta\downarrow 0.
\]
In particular, $\overline{\vol}_s(Z^\ast)=\infty$ at $\gamma_d^\ast=d-1$. Moreover, we can consider 
$
\overline{\vol}_{s,R}(Z^\ast)\defeq \E\big[\mathcal{H}^d\big(Z^\ast\cap \BB(\sfp,R)\big)\big]
$
for $R>0$. Then $\overline{\vol}_{s,R}(Z^\ast)$ satisfies exactly the same large $R$ asymptotics as the visibility set of the Boolean model in Corollary \ref{cor:AtCriticality}.
\end{remark}}

\subsection*{Acknowledgements}
This work has been supported by the DFG priority programme SPP2265 ``Random Geometric Systems''.

\end{document}